\documentclass{amsart}

\usepackage{amssymb}
\usepackage{amsmath}
\usepackage{epsfig}
\usepackage{verbatim}
\usepackage{graphicx}
\usepackage[usenames]{color}

\def\Label#1{\label{#1}}

\let\phi\varphi

\def\A{{\mathbb A}}

\def\Z{{\mathbb Z}}

\def\R{{\mathbb R}}

\def\HTwo{{\mathbb H}^2}
\def\H{{\bf H}}
\def\HH{{\H}}
\def\delH{{\partial\H}}

\def\half{{1\over 2}}

\def\cal{\mathcal}
\def\cala{{\cal A}}
\def\calg{{\cal G}}
\def\calo{{\cal O}}
\def\cald{{\cal D}}

\def\calp{{\cal P}}

\def\qed{{\hfill$\Box$}}

\def\dehn{Cannon's\ }
\def\dehnMachine{Dehn machine}

\newtheorem{theorem}{Theorem}[section]
\newtheorem{lemma}[theorem]{Lemma}
\newtheorem{proposition}[theorem]{Proposition}
\newtheorem{corollary}[theorem]{Corollary}

\newtheorem{definition}[theorem]{Definition}
\newtheorem{question}[theorem]{Question}
\newtheorem{remark}[theorem]{Remark}

\newtheorem{review}[theorem]{Basic properties}

\def\classno#1{}

\title{On a generalization of Dehn's algorithm}

\author{Oliver Goodman}

\author{Michael Shapiro}

\address{
  Department of Mathematics and Statistics,
  University of Melbourne,
  Parkville, Victoria 3052,
  Australia}
\address{
  Department of Pathology,
  Tufts University,
  Boston, MA, 
  USA}

\classno{20F65} % 68Q42 (secondary) grammars and rewriting systems.

\thanks{Michael Shapiro wishes to thank David Thorley-Lawson for
support from Public Health Service grant (RO1 AI062989).}

\date{\today}

\begin{document}

\begin{abstract}
  Viewing Dehn's algorithm as a rewriting system, we generalise to
  allow an alphabet containing letters which do not necessarily
  represent group elements. This extends the class of groups for which
  the algorithm solves the word problem to include nilpotent groups,
  many relatively hyperbolic groups including geometrically finite
  groups and fundamental groups of certain geometrically decomposable
  manifolds. The class has several nice closure properties. We also
  show that if a group has an infinite subgroup and one of exponential
  growth, and they commute, then it does not admit such an algorithm.
  We dub these \dehn algorithms.
\end{abstract}

\maketitle

\def\lh{left-hand}
\def\rh{right-hand}
\def\lhs{left-hand side}
\def\rhs{right-hand side}

\section{Introduction}

\subsection{Dehn's algorithm}
Early last century Dehn \cite{D} introduced three problems.  We know
them now as the word problem, the conjugacy problem and the
isomorphism problem.  Given a finitely generated group $G$ and
generating set $\calg$, we have solved the word problem if we can give
a procedure which determines, for each word $w\in\calg^*$ whether or
not $w$ represents the identity.  We have solved the conjugacy problem
if we can give a procedure which determines, for each pair of words
$u, v \in \calg^*$,
whether they represent elements which are conjugate in $G$.  For the
isomorphism problem, Dehn invites us to develop procedures for
determining if two given groups are isomorphic.

Using hyperbolic geometry Dehn proceeded to solve the word and
conjugacy problems for the fundamental groups of closed hyperbolic
surfaces.  Let us take a moment to describe his solution of the word
problem.  For specificity, let us take the two-holed surface group
\[\langle x_1,y_1,x_2,y_2 \mid [x_1,y_1][x_2,y_2]\rangle.\] The Cayley
graph of this group sits in $\HTwo$ as the 1-skeleton of the
tessellation of $\HTwo$ by regular hyperbolic octagons, and the relator
$R=[x_1,y_1][x_2,y_2]$ labels the boundary of each octagon.  A word
$w$ now lies along the boundaries of these octagons and is a closed
curve if and only if it represents the identity.  Dehn then shows that
any reduced closed curve travels around the far side of some
``outermost'' octagon and in doing so contains at least 5 of its 8
edges.  That is, each reduced word representing the identity contains
more than half of a relator.  (Here we are allowing cyclic
permutations of $R$ and $R^{-1}$.)

This solves the word problem, for we can decompose the relator as
$uv^{-1}$ where $u$ appears in $w=xuy$ and $u$ is longer than $v$.
 This allows us to replace $w$ with the shorter word $w'=xvy$.  If
the word $w$ represents the identity and $w'$ is not empty, we can
again shorten $w'$ in similar manner.  This process either ends with
a non-empty word which we cannot shorten, in which case $w$ did not
represent the identity, or with the empty word in which case $w$ did
represent the identity.

Accordingly, we say the the group $G$ has a {\em Dehn's algorithm} if
it has a finite presentation 
\[\langle \calg \mid \cald \rangle\]
such that every word $w \in \calg^*$ representing the identity
contains more than half of some relator in $\cald$.  Equivalently, we
could write $\cald$ as a finite set of relations $u_i = v_i$ so that
for each $i$, $\ell(u_i) > \ell(v_i)$ and every word $w\in\calg^*$
representing the identity contains some $u_i$.

It is a theorem \cite{L}  \cite{C1} \cite{ABC+} that a group has such a
Dehn's algorithm if and only if it is one of those groups which are
variously called Gromov hyperbolic, hyperbolic, negatively curved or
word hyperbolic.

\subsection{A new definition}
Cannon \cite{C2} suggested we take the following viewpoint.  We have a
class of machines designed to carry out\footnote{\ Morally, Dehn's
algorithm represents a linear time solution to the word problem, but
this actually depends on the machine implementation.  If it is
implemented on a classical one-tape Turing machine, the running time
is $\calo(n^2)$ due to the need to exorcize (or traverse) the blanks
left by each replacement.  If it is implemented on a random access
machine, it is $\calo(n\log n)$ due to the size of the words needed to
indicate addresses.  If it is implemented out on a multi-tape machine
it is $\calo(n)$ since here blanks ``evaporate'' between the tapes
\cite{DA}.  Recently, \cite{Ho} has shown that there is a real-time
multi-tape implementation.} Dehn's algorithm.  Such a machine would be
equipped with a finite set of length reducing replacement rules $u_i
\to v_i$.  It would have a window of finite width through which it
would examine a given word.  This window would start at the beginning
of the word.  As the window moved along, the machine would scan the
word looking for occurrences of $u_i$'s.  If it fails to find any
$u_i$ and is not already at the end of the word, it moves forward.  If
it finds a $u_i$ it replaces it with the corresponding $v_i$.  (The
blank spaces magically evaporate.)  The window then moves backwards
one letter less than the length of the longest $u_i$ or to the
beginning of the word if that is closer.  It accepts a word if and
only if it succeeds in reducing that word to the empty word.

The key difference here is that our working alphabet is no longer
restricted to the group generators.  We shall see that there are
several different classes of machines here with some rather
divergent properties.  We do not know if these competing definitions
for the title of ``\dehnMachine'' yield different classes of groups.
Our most restrictive version solves the word problem in a much larger
class of groups than the word hyperbolic groups.

We describe these classes of machines in terms of rewritings that
they carry out.  In each of these, we are supplied with an alphabet
$\A$ and a finite set of pairs $(u_i, v_i) \in \A^*\times\A^*$ where
for each $i$, $\ell(u_i) > \ell(v_i)$.  We call these {\em rewriting
rules} and write $u_i\to v_i$.  We call $u_i$ and $v_i$ the {\em
left-hand side} and the {\em right-hand side} respectively. For
technical reasons we also have to allow the machines to have {\em
anchored} rules: these are rules which only apply when the left-hand
side is an initial segment of the current word. We write \verb'^'$u$ for
the left-hand side of an anchored rule and consider $u$ and \verb'^'$u$ to
be distinct.

\def\ira{incremental rewriting algorithm}

Let $S$ be a finite set of rewriting rules such that each \lh\ side
appears at most once. We say that $w\in\A^*$ is {\em reduced} with
respect to $S$ if it contains none of the \lh\ sides in $S$.  The
following algorithm, which we call the {\em \ira}\ given by $(\A,S)$,
replaces any $w\in\A^*$ by a reduced word in finitely many steps. If
$w$ contains a \lh\ side, find one which ends closest to the start of
$w$; if several end at the same letter, choose the longest; if
possible, choose
an anchored one in preference to a non-anchored one of the same
length. Replace it
by the corresponding  \rh\ side.  Repeat until $w$ is
reduced.

\def\nira{non-incremental rewriting algorithm}

Here is a slightly different definition: the {\em \nira}\ given by
$(\A,S)$, replaces any $w\in\A^*$ by a reduced word in finitely many
steps. Here $S$ may also include end-anchored rules, with left-hand
side $u$\verb'^', and rules anchored at both ends. 
If $w$ contains a \lh\ side, find one which starts closest to
the start of $w$; if several start at the same letter, choose the
longest; prefer anchored rules when there is a choice.  
Replace it by the corresponding  \rh\ side.  Repeat
until $w$ is reduced.

Each of these algorithms gives a reduction map $R=R_S:\A^*\to\A^*$
where $R(w)$ is the reduced word which the algorithm produces
starting with $w$.  The \ira\ gets its name from the following
property: if $R$ is the reduction map of an \ira, then
$R(uv)=R(R(u)v)$. 

We may wish to apply an \ira\ only to words in $\A_0^*$ where
$\A_0\subseteq\A$. We then refer to $\A_0$ as the {\em input} alphabet
and $\A$ as the {\em working} alphabet. The algorithm can then be
given as a triple $(\A_0, \A, S)$.  We say that $\{ w \in \A_0^* \mid
R(w) {\rm ~is~empty}\}$ is the {\em language} of this triple.  The
same can be done for \nira s\begin{footnote}{Since this work first
appeared in preprint form, Mark Kambites and Friedrich Otto \cite{KO}
have shown that the \ira \ lanaguages are contained in the set of
Chrurch-Rosser languages and that a language is a \nira\ language if
and only if it is a Church-Rosser language.}\end{footnote}.

Clearly Dehn's Algorithm can be implemented as an \ira, with $\A_0 =
\A = \calg$ and $S$ obtained from the $u_i$. We generalize this as
follows. (See Section~\ref{endo} for the example which originally
motivated this definition.)

\begin{definition}
A group $G$, with semi-group generators $\calg$, has a \dehn algorithm if
there exists an alphabet $\A\supseteq\calg$, and set of rewriting rules
$S$ over $\A$, such that the \ira\ reduces $g\in \calg^*$ to the empty
word, if and only if $g$ represents the identity in $G$. 
\end{definition}

We have chosen \ira s because of their nice group theoretic
properties.  Using \ira s in the previous definition ensures that the
\dehn algorithm remembers group elements.  That is, if $G$ has a
\dehn algorithm with input alphabet $\calg$ and reduction map $R$,
and there are $x$ and $y$ in $\calg^*$ so that $R(x)=R(y)$, then $x$
and $y$ denote the same element of $G$.  This property does not hold
in general if one uses \nira s.

On the other hand, \nira s have nice language theoretic properties in
that they support composition.  In the following, we will conceal some
technical details in the word ``mimics''.  One can imagine the \nira\ 
as being carried out by a machine with a finite number of internal
states $s_i$ and a list of rewriting rules $S_i$ for each state
$s_i$.  There is a \nira\ which mimics the action of this multi-state
machine.  Consequently, given two \nira s over the same alphabet $\A$
with reduction maps $Q$ and $R$, there is a \nira\ which mimics a
\nira\ whose reduction map is $R\circ Q$.  We will refer to a \dehn
algorithm carried out using a \nira\ as a non-incremental \dehn
algorithm. 

\subsection{Results}
Before describing our results, we note that many of these were
independently rediscovered by Mark Kambites and Friedrich Otto
\cite{K}.  We show here that groups with \dehn algorithms have the
following closure properties:
\begin{enumerate}
\item{} If $G$ has a \dehn algorithm over one finite generating set then
it has a \dehn algorithm over any finite generating set.
\item{} If $G$ has a \dehn algorithm and $G$ is a finite index
subgroup of $H$ then $H$ has a \dehn algorithm.
\item{} If $G$ and $H$ have \dehn algorithms, then $G*H$ has a \dehn
algorithm. 
\item{} If $G$ has a \dehn algorithm and $H$ is a finitely generated
subgroup of $G$ then $H$ has a \dehn algorithm.
\end{enumerate}

This last closure property significantly increases the class of groups
with \dehn algorithms.  Every word hyperbolic group has a \dehn
algorithm, and as Bridson and Miller have pointed out to us, the
finitely generated subgroups of word hyperbolic groups include groups
which are not finitely presented and groups with unsolvable conjugacy
problem \cite{Subgps}.

We also show that groups with \dehn algorithms include
\begin{enumerate}
\item{} finitely generated nilpotent groups, 
\item{} many relatively hyperbolic groups including geometrically
  finite hyperbolic groups, and fundamental groups of graph manifolds
  all of whose pieces are hyperbolic. 
\end{enumerate}

We prove the first of these by means of expanding endomorphisms.  The parade
example of an expanding endomorphism is the endomorphism of the integers $n
\mapsto 10n$.  The facts that this map makes everything larger and
that its image is finite index combine to give us decimal notation.
Our \dehn algorithms for nilpotent groups consist of this sort of
decimalization together with cancellation.  We are then able to
combine these methods with the usual word hyperbolic \dehn
algorithms to produce the second class of results.

We are also able to prove that many groups do not have \dehn
algorithms.  We have the following criterion: suppose $G$ has two
subsets, $S_1$ and $S_2$ and that both of these are infinite and the
growth of $S_2$ is exponential.  Suppose also that these two sets
commute.  Then $G$ does not have a \dehn algorithm.  This allows us
to rule out many classes of groups including Baumslag-Solitar groups,
braid groups, Thompson's group, solvegeometry groups and the
fundamental groups of most Seifert fibered spaces.  In particular, we
are able to say exactly which graph manifolds have fundamental groups
which have \dehn algorithms.

We have discussed \dehn algorithms which are carried out by \ira s and
\nira s.  They can also be carried out non-deterministically.  Given a
finite set of length reducing rewriting rules, these solve the word
problem nondeterministically if for each word $w$, $w$ represents the
identity if and only if it can be rewritten to the empty word by some
application of these rules.  All of these competing versions are
closely related to the family of {\em growing context sensitive
languages}.  A growing context-sensitive grammar is one in which all
the productions are strictly length increasing.  It is a theorem that
a language $L$ is a growing context-sensitive language if and only if
there is a symbol $s$ and a set of length reducing rewriting rules
such that a word $w$ is in $L$ if and only if it can be rewritten to
$s$ by some application of these rules.  While the family of languages
with non-deterministic \dehn algorithms and the family of growing
context-sensitive languages may not be exactly the same, our criterion
for showing that a group does not have a \dehn algorithm also seems
likely to show that its word problem is not growing context-sensitive.
Now all automatic groups (and their finitely generated subgroups) have
context-sensitive word problems \cite{Sh}.  Thus extending this result
to the non-deterministic case would show that the class of groups with
growing context-sensitive word problem is a proper subclass of those
with context-sensitive word problem\begin{footnote}{Examples of groups
with context-sensitive word problem, but not growing context-sensitive
word problem are given in \cite{KO}. In work in progress (joint with
Derek Holt and Sarah Rees) we show that a language is growing
constext-sensitive if and only if it is the language of a
non-deterministic \dehn algorithm.  In addition, we show that the
methods of Sections \ref{splicingEtc} and \ref{noDehn} extend to these
non-deterministic \dehn algorithms. This has additional
language-theoretic consequences.}\end{footnote}.

\subsection{Thanks} We wish to thank Gilbert Baumslag, Jason
Behrstock, Brian Bowditch, Martin Bridson, Bill Floyd, Swarup Gadde,
Bob Gilman, Susan Hermiller, Craig Hodgson, Chuck Miller, Walter
Neumann and Kim Ruane for helpful conversations.  We also wish to give
special thanks to Jim Cannon for suggesting the key idea of this work
to us during a conference at the ANU in 1996, and for working with us
during the evolution of this paper.

\def\cent{\vert\!\! c }
\def\bl{{\rm b}}

\section{Basic Properties} \label{defs}

\noindent
Let us start by justifying the term \ira.

\begin{lemma} \label{incremental}
Let $R:\A^*\rightarrow \A^*$ denote reduction by a fixed \ira. Then
for all $u,v\in\A^*$, $R(uv) = R(R(u)v).$
\end{lemma}

\begin{proof}
If a substitution can be made in $u$, the same substitution will be
made in $uv$. Therefore, in exactly the number of steps the algorithm
takes to change $u$ into $R(u)$, it changes $uv$ into $R(u)v$. 
This shows that $R(u)v$ is an intermediate result of running the
algorithm on $uv$. It follows that both must reduce to the same
eventual result i.e., $R(uv)=R(R(u)v)$. 
\end{proof}

\begin{proposition} \label{remembers}
Let $R$ denote reduction with respect to a 
\dehn algorithm $(\calg,\A,S)$
for $G$. Let $x,y$ be words in $\calg^*$ such that $R(x) = R(y)$.  Then $x$
and $y$ represent the same element of $G$.
\end{proposition}

\begin{proof}
If $R(x)=R(y)$ then $R(x)y^{-1} = R(y)y^{-1}$ from which it follows,
by Lemma~\ref{incremental}, that $R(xy^{-1})$ equals the empty
word. But since $R$ comes from a \dehn algorithm, this implies that $x$
and $y$ represent the same group element.
\end{proof}

This means that a \dehn algorithm always remembers what element of
the group it was fed.  In a sense this tells us that $R(x)$ is a kind
of ``canonical form'' for $x\in \calg^*$. 

As we shall see, Proposition~\ref{remembers} does not hold for
non-incremental \dehn algorithms.  The following proposition shows
that the incremental rewriting algorithms form a subclass of the
non-incremental ones.

\begin{proposition} \label{generalizes}
Given rewriting rules $(\A, S)$ there is a set of rewriting rules
$(\A,S')$ such that the non-incremental rewriting algorithm of
$(\A,S')$ carries out exactly the same substitutions as the incremental
rewriting algorithm of $(\A, S)$. 
\end{proposition}

\begin{proof}
Suppose we carry out the non-incremental rewriting algorithm given by
$(\A, S)$. In what situation would it make a different substitution to
that chosen by the \ira? Clearly only when we encounter nested
left-hand sides in our word. In that case the non-incremental
algorithm chooses the longer word because it starts first, whereas the
incremental algorithm chooses the shorter because it ends first. But
this means that the \ira\ will never actually invoke the rule with the
longer \lh\ side. Therefore we can discard from $S$ any rules whose
\lh\ sides contain another \lh\ side ending before the last letter. 
Call the set of rules we obtain $S'$. Using these rules both
algorithms make exactly the same substitutions.
\end{proof}

%%%%%%%%%%%%%%%%%%%%%%%%%%%%%%%%%%%%%%%%%%%%%%%%%%%%%%%%%%%%%%%%

\subsection{Rewriting algorithms and compression}
The key result underlying the group theoretic properties of \dehn
algorithms is that if a group has a \dehn algorithm with respect to
one (finite) set of generators, it has one with respect to any other.

\def\astarn{{\A^{*n}}}
\def\gas{(\calg,\A,S)}

Let $\calg$ and $\calg'$ be sets of semi-group generators for $G$, such that
$(\calg,\A,S)$ is a \dehn algorithm for $G$. Each element of $\calg'$ can
be expressed as a word in $\calg^*$. Let $n$ be the length of the longest
such word. Let $\A^{*n}$ be the set of non-empty words of length at
most $n$ in $\A^*$. We can use it as an alphabet, each of whose
letters encodes up to $n$ letters of $\A$. Since $\calg\subseteq\A$ we
can regard $\calg'$ as a subset of $\astarn$.

The {\em writing out} map from $(\astarn)^*$ to $\A^*$ maps a word to
the concatenation of its letters.  Lemma~\ref{compression} shows that
given $\gas$ we can construct an algorithm $(\calg', \astarn, S')$ which,
by ``mimicking'' $\gas$, deletes its input precisely when $\gas$
deletes the written out version of the same input.  Unfortunately the
algorithm we give is not quite an \ira: its rules are not strictly
length decreasing. The main point of this section is to explain how we
can overcome this problem and give an \ira\ which does what we
want. 

\begin{lemma} \label{compression}
Let $(\A_0,\A,S)$ be an incremental (or non-incremental) rewriting
algorithm. Then for any integer $n>0$ there exists a non-strictly
length decreasing incremental (resp.\ non-incremental) rewriting
algorithm $(\A_0^{*n},\astarn,S')$ with the following property. For
each word $w\in(\A_0^{*n})^*$, the reduction of $w$ with respect to
$(\A_0^{*n},\astarn,S')$ written out, equals the reduction with
respect to $(\A_0,\A,S)$ of $w$ written out.
\end{lemma}

\begin{proof}
Let $W$ be the length of the longest \lh\ side in $S$. 

For an incremental algorithm, 
the set of \lh\ sides in $S'$ 
is the set of all words of length less than
or equal to $W$ in $(\astarn)^*$, with and without leading \verb'^' 's,
which, when written out, contain a \lh\
side of $S$. For each such word, we write it out, apply one
substitution from $(\A_0,\A,S)$, and write it back into $(\astarn)^*$
to obtain the corresponding \rh\ side; an anchored rule can only be applied
if the \lh\ side starts with a \verb'^'.

That this can be done without making the \rh\ side any longer in
$(\astarn)^*$ than the \lh\ side should be clear: one case
when the \rh\ side cannot be any shorter is when the \lh\ side is one
letter long, and the substitution we make on the written out word does
not entirely delete it. 

We have to check that, modulo writing out, the two algorithms carry
out the same substitutions. Let $w$, written out, contain a \lh\ side
$u$ of $S$. 
Some subword of $w$, adorned with a \verb'^' if it is an
initial segment, contains $u$, and is a \lh\
side in $S'$. 
The first $S'$-\lh\ side can't end to the left of the
end of $u$, since it would then contain no $S$-\lhs\
at all. Therefore the first $S'$-\lh\ side contains
$u$, and is anchored if $u$ is an initial segment. The rule in $S'$
for this \lh\ side carries out the substitution in $S$ for $u$. 

For a non-incremental algorithm, the set of \lh\ sides in $S'$ 
is the set of words $U\in (\astarn)^*$ of length less than
or equal to $W$, with optional leading and trailing
\verb'^' 's, such that 
\begin{enumerate}
\item $U$ written out contains a \lh\ side of $S$, and
\item if the first $S$-\lh\ side in $U$
starts fewer than $W$ $\A$-letters from the end of $U$, then 
$U$ ends with a \verb'^'. 
\end{enumerate}

Let $w$ and $u$ be as above. 
We can find an $S'$-\lh\ side $U$ in $w$ which contains $u$.
Now $u$ could have a subword $u_0$ which is also a
\lh\ side in $S$. In principle, the first $S'$-\lh\ side in $w$ might
contain $u_0$ but not $u$, but this is ruled out by (2). 
Therefore the first $S'$-\lh\ side in $w$ contains $u$,
and the corresponding rule does the right substitution. 
\end{proof}

We want to adjust this basic construction so as to obtain rules which
are strictly length decreasing.  If each rule in $S$ were to delete at
least $n$ letters, there would be no problem, but this will not
generally be the case.  When the input word has two or more letters we
might write the result of a single substitution as a shorter word in
$(\A^{*(2n-1)})^*$. This doesn't really solve the problem since we end
up working in larger and larger alphabets. And what about a word of
length $1$ in $(\astarn)^*$ which when written out and reduced, is
non-empty? The algorithm we construct will not touch such a word
unless it can delete it entirely. In fact, unless it can delete its
input completely, it may stop short with some intermediate result of
the original algorithm. This is fine since we only really care
whether or not an input word is deleted completely. 
% Our solution combines all of these ideas.

We return first to the original algorithm $(\A_0,\A,S)$, and try to see
to what extent it can be made to remove several letters at a time when it
substitutes. 

It is helpful to think of the algorithm as being
carried out by a machine which views the word it is processing through
a window of size $W$, where $W$ is the length of the longest \lhs\ in
$S$. 
% % keeps a current position in the word it is processing.  
Since the incremental algorithm works by observing the
earliest ending left hand side, one might imagine that the machine acts
when a left hand side ends at the end of the window. 
Similarly, a \nira\ acts when a \lh\ side starts at the start of the
window. 
% We define the
% {\em current position} of the machine to be the position of the \rh\
% edge of the window. 
% In the case of a \nira, the current position is the position of the
% \lh\ edge of the window. 
If there are no \lh\ sides visible,
% ending at the current position, (or starting there in the case of a \nira) 
the machine steps one letter to the right,
or stops if it has reached the end of the word. If there {\em is} a \lh\
side, it substitutes and steps $W-1$ letters to the left.

Let $w$ be a word containing a \lhs\ and let $u$ be the first
such in $w$.  Let us look at a subword $U$ of $w$ extending $A\geq W$
letters to the left of $u$, and $B\geq W$ letters to the right, and
see what the machine does. The machine's actions are entirely
determined by the contents of this {\em $A,B$-neighborhood of $u$}
until such time as it needs to examine letters either to the left or
to the right of it.  We say that the machine goes to the left
or to the right accordingly.  In the first case the machine must first
make at least $\lfloor \frac{A}{W-1} \rfloor + 1$ substitutions.  We
call each substitution made in this way a {\em subword reduction.}

If there are fewer than $A$ letters to the left of $u$, $U$ is an
initial segment of $w$; then the machine's actions are
determined by the contents of $U$ until it (inevitably) goes to the right. If
there are fewer than $B$ letters to the right, the machine can either
go to the left or terminate.

We make rules which carry out several substitutions at a time.
The new \lh\ sides are the reducible words with no more than $A$
letters before the first \lh\ side, and no more than $B$ letters after
it.  The new \rh\ sides are the result of running the machine on the
\lh\ sides until it goes to the left or the right. If the new \lh\
side has fewer than $A$ letters before its first $S$-\lh\ side, we allow
the machine to run until it goes to the right and make the resulting
rule be anchored at the start. If there are fewer than $B$ letters after the
$S$-\lh\ side, we allow the machine to run until it goes to the
left or terminates; for a \nira\ we make the resulting rule be
end-anchored. We call the rules we obtain {\em left-going} if
the machine went to the left and {\em right-going} if it went to the
right or terminated. 

% A machine using these new rules does exactly what the original machine
% did, but several steps at a time. If it applies a left-going rule it
% deletes at least $\lfloor \frac{A}{W-1} \rfloor + 1$ letters. If it
% applies a right-going rule, the current position 
% shifts to the right. 

Finally, let us discard all right-going rules which have a \lh\ side
with fewer than $B$ letters after the first $S$-\lh\ side, and
non-empty \rh\ side. Let $S'$ contain all the remaining rules.  We
claim that as long as $A \geq B + W$ and $B > W-1$ a machine using the
rules $S'$ still carries out the same substitutions but may stop short
of fully reducing the input word (with respect to $S$).

Let $w$ be a word containing an $S$-\lhs\ and let $u$ be the first
such in $w$. We have to show that if $w$ contains an $S'$-\lhs\ then
the first such contains $u$. An $S'$-\lhs\ can't end to the left of the
end of $u$ since in the incremental case it would contain no $S$-\lhs,
while in the non-incremental case it would have to be a
non-end-anchored rule with fewer than $B$ letters to the right of its
first $S$-\lhs. Therefore {\em if} any $S'$-\lhs\ contains $u$, the first
one in $w$ does. 

If we can find no $S'$-\lhs\ containing $u$ then the
$A,B$-neighborhood of $u$ must be one of the deleted \lh\ sides. In
that case $u$ ends within $B$ letters of the end of $w$. Since $A \geq
B+W$, any other rule which might apply, containing some other
$S$-\lhs\ $u_1$ to the right of $u$, would also see $u$, which is a
contradiction. The fact that the rule for the
$A,B$-neighborhood of $u$ has been deleted means that in this case the
original algorithm would have terminated with a non-empty result.

With $w$ and $u$ as above, we define the {\em reduction point} of an \ira\
to be the \rh\ edge of $u$, while for a \nira\ it is the \lh\ edge of
$u$. 
Each rule is either, 
\begin{enumerate}
\item left-going, deleting at least $\lfloor \frac{A}{W-1} \rfloor + 1$
letters, 
\item right-going, deleting the whole \lh\ side, or
\item right-going, shifting the reduction point 
at least $B-(W-1)$ letters to the right, or out of the word entirely. 
\end{enumerate}

\begin{lemma} \label{compression2}
Let $(\A_0,\A,S)$ be an incremental (or non-incremental) rewriting
algorithm. Then for any integer $n>0$ there exists an incremental
(resp.\ non-incremental) rewriting algorithm
$(\A_0^{*n},\A^{*(2n-1)},S')$ with the following property. For each
word $w\in(\A_0^{*n})^*$, the reduction of $w$, with respect to
$(\A_0^{*n},\A^{*(2n-1)},S')$, written out is an intermediate result
of the reduction of $w$ written out with respect to $(\A_0,\A,S)$.  It
is empty if and only if the latter is also.
\end{lemma}

\def\cB{{\cal B}}
\def\cC{{\cal C}}

\begin{proof}
Let us first give names to parts of our new working alphabet. Let
$\cB$ be all words in $\A^{*(2n-1)}$ of length at most $n$, and let
$\cC$ be all longer words. Our new input alphabet is a subset of
$\cB$, and an input word is a word in $\cB^*$. 

At any given time during the running of our new algorithm the current
word will satisfy the following conditions. No $\cC$-letters end (in
the written out word) to the right of the reduction point. Any
$\cC$-letters present will end at least $(2n-1)(W-1)$ original letters
apart, i.e.\ they will be relatively sparse.

We shall give rules that, modulo writing out, carry out subword
reduction looking at least $A = 2nW + W$ original letters to the left of
the first \lh\ side and $B = 2nW$ letters to the right. 
% The rules in $S'$ are as follows. 
The left-hand sides are words in
$(\cB\cup\cC)^*$ such that 
\begin{enumerate}
\item each is $S$-reducible when written out, 
\item each has up to $A + (2n-2)$ original letters before the
first original \lh\ side and up to $B + (n-1)$ following it,
\item any $\cC$-letters present come before the reduction point and are
sparse, as noted above,
\item if there are fewer than $A$ original letters before the
first original \lh\ side, it starts with a \verb'^', and
\item (non-incremental case only) if there are fewer than $B$ original
letters following the first original \lh\ side, it ends with a \verb'^'. 
\end{enumerate}
Modulo writing out, these are the same \lh\ sides as before except that we
have to allow for the granularity of the $\cB$ and $\cC$ letters. 

% We can redefine the $A,B$-neighborhood of an $S$-\lhs\ to be the shortest
% subword of a word in $(\cB\cup\cC)^*$ that, when written out, contains the
% $A,B$-neighborhood as originally defined. 

To obtain each corresponding \rh\ side we apply subword reduction to
the written out word for $2n-1$ steps or until subword reduction is
complete if this happens first: it follows that there will be no
left-going rules.
If $2n-1$ substitutions were made (or the \lh\ side was deleted
entirely) we can write the result using at least one fewer
$\cC$-letters, or fewer $\cB$-letters if no $\cC$-letters were
present. Since the reduction point moves at most $(2n-1)(W-1)$
original letters to the left, it moves past at most one
$\cC$-letter. Therefore we can write our \rh\ side so as to preserve
the above conditions on the placement and sparsity of
$\cC$-letters. 
% If there were fewer than $2nW + W$ original letters to
% the left of the original \lh\ side, we make the rule be anchored to
% the start of the word.

If subword reduction is complete before $2n-1$ substitutions have been made,
and the result is non-empty, it may be impossible to keep the number of $\cC$
letters fixed and still write a length reducing rule. If this is the
case, and there were fewer than $B$ original letters
after the original \lh\ side, we discard the rule entirely. With
$B > 2n-1$ or more original letters after the \lh\ side, only
reductions which remove fewer than $n$ letters can force us
to introduce a new $\cC$-letter. For an \ira\ the new reduction point 
will be to the right of our subword. 
By writing the new $\cC$-letter at
the end of the \rh\ side we ensure that it ends at least $(2nW -
(n-1))$ letters to the right of the previous reduction point. Since 
$(2nW - (n-1)) > (2n-1)(W-1)$ the sparsity of $\cC$-letters is preserved.
For a \nira\ the new reduction point could be up to $W-1$ letters
in from the end of our \rh\ side. Thus our new $\cC$-letter might have
to end up to $W-1+n-1$ original letters from the end of the \rhs. 
This still puts it at least $(2nW - 2(n-1) - (W-1)) > (2n-1)(W-1)$
letters to the right of the previous reduction point.

The rules we have given are strictly length decreasing. They preserve
the conditions given on the placement of $\cC$-letters.  Modulo
writing out and working several steps at a time, the rules apply the
same substitutions as the original algorithm. If a word is reducible
when written out, either a rule will apply, or the word will be a few steps
away from being reduced with a non-empty result. It follows that the
new rules delete a word in $(\A_0^{*n})^*$ if and only if the original
rules deleted the same word written out. 
\end{proof}

%%%%%%%%%%%%%%%%%%%%%%%%%%%%%%%%%%%%%%%%%%%%%%%%%%%%%%%%%%%%%%%%%

\subsection{Composition of non-incremental rewriting algorithms}

Let us introduce the notion of a finite state \dehnMachine. As with
rewriting algorithms these can be either incremental or
non-incremental. (We describe the non-incremental version: to obtain
the incremental version, read ``ending at the current position''
wherever the definition says ``starting at the current position.'')
Such a machine comes with a finite collection of
states, $Q=\{q_i\}$.  One of these, $q_0$ is the start state. For each
state $q\in Q$ there is a collection of length reducing replacement
rules $S_q = \{u_i \to v_i\}$.  There is also a transition function
which chooses a new state depending on the current state and the
contents of the subword of length $W$ starting at the current
position, where $W$ is an upper bound for the lengths of all the \lh\
sides.

Such a machine starts in the start state at the beginning of the input
word. In state $q$, it looks at the next $W$ letters for the longest
\lhs\ in $S_q$ starting at the current position, and to determine its
new state. It then either substitutes and steps $W$ letters to the
left, or steps one letter to the right. In either case it switches to
the new state.  It terminates when it reaches the end of the word with
no further replacements possible.

Observe that when a \dehnMachine\  with state terminates it does not
necessarily leave behind a word which is free of \lh\ sides. While
\dehnMachine s with state are ostensibly more powerful than rewriting
algorithms, we show that, by storing the state information in the
current word, we can get a rewriting algorithm to ``mimic'' a
\dehnMachine. We then use \dehnMachine s with state to show that 
\nira s have a nice composition property. 

We can extend the concept of {\em writing out} to include any map
$\A^{\prime *}\rightarrow \A^*$ induced by a map $\A'\rightarrow \A^*$. 
A machine {\em stops short} if it terminates at a point when all
remaining substitutions would have applied to a final segment of
bounded length. One machine {\em mimics} another if the result of the
mimic written out is always a result of the original stopping short. 

\begin{proposition}\label{statemachine}
Given a non-incremental (or incremental) finite state \dehnMachine,
there is a \nira\ (resp.\ \ira) which mimics it. The mimic terminates 
with an empty word if and only if the finite state \dehnMachine\ 
terminates with an empty word {\em in its start state.}
\end{proposition}

\begin{proof}
We give first a non-strictly length decreasing rewriting algorithm.
At the end we sketch how the trick used in the proof of
Lemma~\ref{compression2} of introducing widely spaced ``multi-letter''
letters allows us to give strictly length-decreasing rules. 
% As before,
% the mimic algorithm may stop short but terminates
% with the empty word if and only if the finite state \dehnMachine\ 
% terminates {\em in its start state} with an empty word. 
The reason we
prefer to give a non-strictly length decreasing algorithm here is
that, while the details of making strictly length-decreasing rules are
not hard, they would obscure the basically simple idea behind this
proof.

Let $\A$ be the working alphabet of our \dehnMachine. We make copies
of $\A$ in different colors, one corresponding to each state of the
machine, and one more in white. The input alphabet, and the copy of
$\A$ corresponding to the start state, we color indigo. At any given
time during the running of the mimic algorithm an initial segment
(possibly empty) of the current word is white. The first colored
letter indicates a state of the \dehnMachine\  and its current position,
and the remaining letters are all indigo.

Let $W$ be the length of the longest \lhs\ of the \dehnMachine.  We
specify the substitutions we wish the mimic to make rather than giving
the precise rules.
Look $W$ letters to either side of the first colored letter. If a
substitution is indicated (according to the state of the first colored
letter) we make it, color up to $W-1$ letters indigo, and one the
color of the new state.  If no substitution is indicated, the first
colored letter is turned white and the next letter is colored with the
new state.
% As well as full-length
% rules we will need start-anchored rules with fewer than $W$ white
% letters in them and, for the non-incremental case, end-anchored rules
% with fewer than $W$ indigo letters. 
A special case arises for rules which delete
their whole \lh\ side and do not lead to the start state. 
Since there is no suitable letter to color with the new state
the mimic instead writes a colored blank. 
% If the new state is the start state, the \rh\
% side is empty. If the new state is not the start state we must instead
% write a colored blank to keep track of the new state. 
We then have to
add a few more rules which take a colored blank followed by a letter
and write the same letter in that color.

It is not hard to see that the mimic and the \dehnMachine\  make
essentially the same substitutions. When the mimic terminates it is
with a word that is white except for its final letter which indicates
the termination state of the \dehnMachine. If the \dehnMachine\ 
terminates with an empty word, in a non-input state, the mimic leaves
behind a single colored blank. 

To make these rules length decreasing we instead look $2W$ letters
before and after the first colored letter in the incremental case
($2W$ before and $3W$ after if non-incremental).  We run the
\dehnMachine\ as a subword reduction. If no substitutions are made,
the first colored letter is shifted at least $2W$ letters to the right
and two white letters are replaced by one encoding them both. If
subword reduction goes to the left, enough substitutions will be made
to allow us to remove any ``double'' letters we find on the way (these
ending at least $2W$ original letters apart). If we can't see $2W$
(resp.\ $3W$) letters to the right we may have to discard the relevant
rule and allow the mimic machine to terminate a little
prematurely. This only happens in cases where the machine is unable to
make any further substitutions between the current point and the end
of the word.
\end{proof}

Let $(\A,S)$ and $(\A,S')$ be \nira s with reduction maps $P$ and $P'$
respectively. Ideally there would then be a \nira\ with reduction map $P'
\circ P$. Unfortunately this doesn't appear quite to be the case. We
have to allow the resulting algorithm to give its answer in some
``compression alphabet'' $\A^{*n}$, and we may have to allow it to stop
short of reaching its answer. We don't really mind the compression
alphabet, but having a machine stop short is a problem: it gets
in the way of doing any further composition. 

Reluctantly, we must add a further ``flavor'' of \nira\ to our
collection. An {\em nearly strict \nira}\ is one which may include
some length preserving {\em ending rules:} these are end anchored
rules such
that the
resulting algorithm has the property that one of these
will apply only when
the word is reduced with respect to all the
strictly length decreasing rules, and afterwards the word will 
be fully reduced.
We shall not consider here the
question of how to determine, in general, whether a given set of rules
has this property. What is hopefully clear is that if, in the proof of
Lemma~\ref{compression2}, we put back the deleted rules, we obtain a
nearly strict \nira. Modulo writing out, the resulting algorithm
achieves the reduction map of the original algorithm. Furthermore it makes
no difference to the proof if the original algorithm is itself
nearly strict. 

\begin{lemma} \label{compression3}
Let $(\A_0,\A,S)$ be a nearly strict \nira. 
Then for any integer $n>0$ there exists a nearly strict \nira\ 
$(\A_0^{*n},\A^{*(2n-1)},S')$ with the following property. For each
word $w\in(\A_0^{*n})^*$, the reduction of $w$, with respect to
$(\A_0^{*n},\A^{*(2n-1)},S')$, written out is
the reduction of $w$ written out with respect to $(\A_0,\A,S)$. \qed
\end{lemma}

Similarly, when we construct a mimic for a non-incremental
\dehnMachine\  with state, we can avoid stopping prematurely by allowing
ending rules for the resulting machine. We can also allow a
\dehnMachine\ with state to have ending rules. These are length
preserving rules which put it into a {\em terminal state,} a state
without rules which the machine cannot leave. Such a machine can also
be mimicked by a nearly strict \nira, the proof being virtually
unchanged.

We shall show that it is possible to compose nearly strict \nira s. We
can always recover a genuine \nira\, which might stop short, by
discarding the length preserving rules.

\begin{proposition} \label{composition}
Let $(\A,S)$ and $(\A,S')$ be nearly strict \nira s with reduction
maps $P$ and $P'$ respectively. There is a nearly strict \nira\
$(\A,\A',T)$ which mimics the process of first applying $(\A,S)$ and
then applying $(\A,S')$. The reduction map of $(\A,\A',T)$ written out
is the composition $P' \circ P$.
\end{proposition}

\begin{proof}
This process can be carried out by a finite state \dehnMachine.  In
its initial state it applies the rules in $S$. Once no more rules
apply and it approaches the end of the word, it switches to a second
state. In this state it simply compresses a little bit until it
arrives at the start of the word again. Then it switches into a third
state where it uses the rules in $S'$ modified, as in
Lemma~\ref{compression3}, for compressed input.

What if $S$ includes ending rules? Without loss of generality,
$S$-\lh\ sides are either $W>3$ letters long, or anchored at both
ends. Ending rules of length $W$ can be combined with compression. 
Rules anchored at both ends can be modified so as to complete the entire
reduction ($P'\circ P$) at a single step. 

With $S$ as above, our machine can recognize the end of a reduced
word by finding any word of $W-1$ letters which is anchored at the end
but not the start. (Shorter entire words being already dealt with.) It
can then start backtracking and compressing.

When the modified $S'$ has ending rules, these become ending rules for
the finite state machine. Finally, we transform the
resulting \dehnMachine\  with state into a nearly strict \nira. 
\end{proof}

\subsection{Group theoretic consequences}
From Lemma~\ref{compression2} and the discussion at the start of this
section we have the following result. 

\begin{theorem}\label{changegens}
Let $G$ be a group with finite semi-group generating sets $\calg$ and
$\calg'$. Then $G$ has a \dehn algorithm with respect to $\calg$ if and only
if it has one with respect to $\calg'$. \qed
\end{theorem}

\begin{theorem}\label{subgroups}
Let $G$ be a group and let $H$ be a finitley generated subgroup of $G$. 
If $G$ has a \dehn algorithm, $H$ has one too. 
\end{theorem}

\begin{proof}
Choose a set of generators for $G$ which includes generators for $H$.
With respect to these generators, a \dehn algorithm for $G$ is also one
for $H$. 
\end{proof}

\begin{theorem}\label{finiteindex}
Let $G$ be a group and let $H$ be a finite index subgroup of $G$. 
If $H$ has a \dehn algorithm, $G$ also has one. 
\end{theorem}

\def\cH{{\cal H}}
\def\cT{{\cal T}}

\begin{proof} 
Fix a transversal $\cT$ for $[H:G]$ from which we omit the
representative of the identity coset. Fix a finite generating set $\calg$
for $G$ containing $\cT$.  Each word $g_1 g_2 g_3$ with $g_i \in G$ is
equal in $G$ to a word of the form $[h] [t]$, for some $h\in H$ and
$t\in \cT$, where the brackets indicate that each letter may be
omitted. If $g_1 g_2 \in \calg^*$ evaluates to an element of $H$, it can
be written as the $0$ or $1$-letter word $[h]$, again for some $h\in
H$.  As $g_1, g_2, g_3$ vary in $\calg$ we obtain finitely many elements
$h\in H$.  Let $\cH$ be a finite generating set for $H$ containing all
non-identity elements obtained in this way and also, all of $\calg\cap
H$.

The above equalities give rules $R$ of the form $g_1g_2g_3\mapsto ht$
etc. We omit any rules with $g_1 \in \cH$.  The \ira\ $(\calg, \calg\cup\cH,
R)$ turns a word in $\calg^*$ into a word in $\cH^*$ followed by at most
two letters from $\calg$ by pushing a coset representitive along
the word. If an input word to this algorithm represents an element of
$H$, the reduced word will be in $\cH^*$.

Let $(\cH, \A, S)$ be a \dehn algorithm for $H$. We claim that $(\calg,
\calg\cup\cH\cup\A, R\cup S)$ is a \dehn algorithm for $G$. The $R$
rules translate the word into a word in $\cH$ followed by a couple of
letters keeping track of the coset. Then the $S$ rules chase along
behind applying $H$'s \dehn algorithm to the word in $\cH$. The effect
is exactly as if we applied $(\calg, \calg\cup\cH, R)$ first, followed by
applying $(\cH, \A, S)$ to the $\cH^*$ part of the result. If an input
word represents the identity in $G$, the first step produces a
representation of the identity in $\cH^*$ and the second deletes
it. An input word which does not represent the identity will reduce,
either to some word containing letters in $\calg-\cH$, if it does not
evaluate into $H$, or otherwise to a non-empty word in $(\cH\cup\A)^*$.
\end{proof}

The previous theorems hold both for both \dehn algorithms and
non-incremental \dehn algorithms.  The last of these suggests a way
to construct a \dehn algorithm using the \nira\ which does not satisfy
Proposition~\ref{remembers}.  Consider the case of $H$ finite index in
$G$.  It is not hard to parlay a \dehn algorithm for $H$ into a
\nira\ which solves the word problem in $G$ but destroys information in the
case where the word is not in the identity coset.  Here is what it
does: given a word $w$, it first transforms this into a word of the
form $ht$ where $h$ is a word in the generators for $H$ (possibly the
empty word) and $t$ is an element of the transversal, and is empty if
and only if it represents the identity coset.  If $t$ is empty, we now
proceed to reduce $h$ according to the \dehn algorithm for $H$.  On
the other hand, if $t$ is not empty, we can proceed to wantonly
destroy the information in $h$.

\begin{proposition}
There is a \nira\ which is not mimicked by any \ira. \qed
\end{proposition}

\begin{theorem} \label{freeprods}
If $G$ and $H$ both have \dehn algorithms, then so does their free
product $G*H$.
\end{theorem}

\begin{proof}
\def\mod#1{{\rm (mod\ }#1{\rm )}}

We suppose that $G_0$ and $G_1$ are groups with \dehn algorithms
$({\calg}_0, \A_0, S_0)$ and $({\calg}_1, \A_1, S_1)$ respectively
and that the alphabets for these are disjoint.
Let
\begin{align*}
    T_0 &= \{ au\to av \mid \hat{}u\to v \in S_0, \ a \in \A_1 \} \cr
    T_1 &= \{ au\to av \mid \hat{}u\to v \in S_1, \ a \in \A_0 \} \cr
    S &= S_0 \cup T_0 \cup S_1 \cup T_1 \cr
    \calg &= \calg_0 \cup \calg_1 \cr
    \A &= \A_0 \cup \A_1.
\end{align*}    
We claim that $(\calg,\A,S)$ is a \dehn algorithm for $G_0 * G_1$.    

To see this, consider a word $x_0\ldots x_n$ consisting of alternating
non-empty words from the alphabets $\calg_0$ and $\calg_1$.  For
simplicity, we will assume that we have numbered the two groups so that
$x_i \in \calg_{i \mod 2}$.  We claim that as long as no $x_i$ evaluates
to the identity, $R(x_0\ldots x_n) = R_0(x_0)\ldots R_n(x_n)$.  (Here we
are using $R$ to denote reduction with respect to $S$ and $R_i$ to
denote reduction with respect to $S_{i\mod 2}$.  Likewise, we will refer
to $S_{i\mod 2}$ as $S_i$ and $T_{i\mod 2}$ as $T_i$.)

This claim is true when $n=0$, for then only the rules of $S_0$ apply.
Suppose now that this claim holds for $n=k$.  We wish to establish it
for the case $n=k+1$.  By induction an intermediate result of the
reduction of $x_0\ldots x_{k+1}$ is $R_0(x_0)\ldots R_k(x_k) x_{k+1}$
and the portion before $x_{k+1}$ is fully reduced.  Further, the
assumption that no $x_i$ evaluates to the identity implies that
$R_k(x_k)$ is non-empty.  Accordingly any further reductions are made
either by a a non-anchored rule of $S_{k+1}$ applying entirely inside
$x_{k+1}$ or by a rule of $T_{k+1}$ applying at the last letter of
$R_k(x_k)$ and the beginning of $x_{k+1}$. Any rule of $T_{k+1}$ changes
only the letters of $x_{k+1}$ and performs exactly as an anchored rule
of $S_{k+1}$ would have done had $x_{k+1}$ been the beginning of a word.
These combine to produce $R_0(x_0)\ldots R_k(x_k)R_{k+1}(x_{k+1})$ as
required. 

In particular if no $x_i$ represents the identity, then $x_0 \ldots
x_n$ does not represent the identity and does not reduce to the empty
word.

Now consider the case in which some $x_i$ represents the identity.  We
take $x_i$ to be the earliest such.  The process of reducing the word
$w=x_0\ldots x_n$ produces $R_0(x_0)\ldots R_{i-1}(x_{i-1})x_i\ldots
x_n$ as an intermediate result.  As before, $S_{i}$ and $T_{i}$ conspire
to reduce $x_i$ as $S_i$ would have done had $x_i$ stood alone.  This
produces $R_0(x_0)\ldots R_{i-1}(x_{i-1})x_{i+1}\ldots x_n$.  But this
is an intermediate result of reducing $w'=x_0\ldots x_{i-1}x_{i+1}
\ldots x_n$.  Furthermore, $w'$ represents the identity if and only if
$w$ represents the identity.  But the free product length of $w'$ is two
less than the free product length of $w$.  Thus we may assume
inductively that $w'$ reduces to the empty word if and only if it
represented the identity and we conclude the same about $w$. 

Since this induction reduces free product length by two, it remains to
check two base cases.  One is when the free product length of $w$ is 0,
and here there is nothing to check.  The second is when the free product
length is 1.  This is just application of the \dehn algorithm in one of
the factor groups. 
\end{proof}

We do not know how to prove this for non-incremental \dehn
algorithms.  This raises the following

\begin{question}
{\rm 
Are there groups with non-incremental \dehn algorithms which do not 
have \dehn algorithms?
}
\end{question}

\section{Groups with Expanding Endomorphism} \label{endo}

\def\tinv{t^{-1}}
\def\lg{\ell_{\calg}}

\noindent
Let $G$ be a finitely generated group with finite set of semi-group
generators $\calg$. Let $\lg$ denote the word metric on $G$ with respect
to $\calg$. We say that a homomorphism $\phi : G\rightarrow G$ is an {\em
expanding endomorphism} if $\phi(G)$ is a finite index subgroup of $G$
and there exists a constant $M>1$ such that $\lg(\phi(g)) \geq
M\lg(g)$ for all $g\in G$. Observe that by taking a suitable power of
$\phi$ we may make $M$ as large as we wish.  By taking a finite set of
coset representatives for $\phi(G)\backslash G$ we see that there is a
constant $K$ such that for all $g\in G$, the distance from $g$ to
$\phi(G)$ is at most $K$. We say that $\phi(G)$ is $K$-{\em dense} in
$G$.

Let $\A$ be the finite alphabet $\calg\cup\{ t, \tinv \}$, where $t$ and
$\tinv$ are letters not in $\calg$. We say that a word $w$ in $\A$ is
{\it balanced} (with respect to $t$) if $w$ has the same number of
$t$'s as $\tinv$'s, and further, every initial segment of $w$ has at
least as many $t$'s as $\tinv$'s. Each balanced word $w$ in $\A$
represents an element of $G$: we define the element represented by $t
w \tinv$ to be $\phi$ applied to the element represented by $w$.

The following rules (assuming $\phi$ is chosen so that both $M$ and
$K$ are sufficiently large) give a \dehn algorithm for $G$. In the
rules: $g$ denotes a word in $\calg^*$, and $g'$ and $g''$ denote
geodesic words in $\calg^*$ such that $g = \phi(g')g''$, and $\ell(g'')$
equals the distance from $g$ to $\phi(G)$.

\begin{enumerate}
\item Replace any non-geodesic word $g$ of length $\ell(g)\leq 2K$ by
an equivalent geodesic word. 
\item If $g$ is geodesic, with $\ell(g) = 2K$, replace $g$ 
by $t g'\tinv g''$, or replace $\tinv g$ by $g'\tinv g''$.
\item If $g$ is geodesic, with $\ell(g)\leq 2K$ and $\ell(g'') = 0$
(i.e. $g \in \phi(G)$), replace $\tinv g$ by $g'\tinv$.
\item Replace $t\tinv$ by the empty word. 
\end{enumerate}

These rules clearly map balanced words to balanced words, and do not
change the element of $G$ represented. It is clear that Rules 1, 3 and 4
are strictly length decreasing. 
For Rule~2 to reduce length we need $\ell(g') + \ell(g'') + 2 <
\ell(g)$. We have $\ell(g) = 2K$, $\ell(g'') \leq K$, and
$\ell(g')\leq \frac{1}{M}(\ell(g)+\ell(g''))$. It follows that Rule~2
will be length decreasing if $3/M + 2/K < 1$. 

\begin{lemma} \label{normalform}
Let $G$, $\A = \calg \cup \{t,t^{-1}\}$, $M$ and $K$ be as above. Let
$w$ be the reduction of a word in $\calg$ with respect to Rules 1-4.
Then $w$ has the form $t^n g_n \tinv \ldots \tinv g_1 \tinv g_0$, or
just $g_0$ ($n=0$), such that:
\begin{enumerate}
\item each $g_i$ is a geodesic word in $\calg$ of length less than $2K$;
\item each $g_i$, for $i<n$, is either in $G - \phi(G)$ or it is empty; 
\item if $n>0$, $g_n$ is not empty. 
\end{enumerate}
\end{lemma}

\begin{proof}
We show first that all $t$'s appear at the start of $w$. Initially
this is vacuously true. The only rule whose application could make
this untrue is 2 since it is the only rule which creates $t$'s.  But
Rule 2 is only applied at the start of the word, or when the
immediately preceding letter is $t$, for otherwise one of Rules 1-3
would apply at least one letter to the left.

Rule 1 ensures that each $g_i$ is geodesic, while Rule 2 ensures
that the length of each $g_i$ is less than $2K$.  Rule 3 ensures that
each $g_i$, for $i<n$, is either in $G - \phi(G)$ or it is empty.
Rule 4 ensures that $g_n$ is not the empty word if $n>0$.
\end{proof}

\begin{theorem}
Rules 1-4 reduce each word in $\calg$ to the empty word if and only if
that word represents the identity element of $G$.
\end{theorem}

\begin{proof}
Let $g = t^n g_n \tinv \ldots \tinv g_1 \tinv g_0$ be the reduction
of a word in $\calg$ representing the identity in $G$. Let $i$ be the
least integer such that $g_i$ is non-trivial. Then if $i<n$, 2 in
Lemma~\ref{normalform} implies that $g$ belongs to a non-$1$ coset of
$\phi^{i+1}(G)$. Therefore $g_i$ is trivial for $i<n$. 

Hence $g_n$ represents the identity in $G$. By 1, $g_n$ is geodesic and
therefore trivial. By 3, $n=0$ and so $g$ itself is trivial. The
converse is clear.
\end{proof}

The process we have just described is essentially that of writing the
decimal expansion of a number.  Indeed, if you apply this to the sum
of $572$ $1$'s, $1+\dots+1$, using the endomorphism $n \mapsto 10 n$
you will get $t^{2}5t^{-1}7t^{-1}2$.  This is nothing but the decimal
$572$ with $t$'s performing the function of place notation.
Unfortunately, our decimal expansions can be a bit perverse.  In
addition to the numerals for the numbers $0$ through $9$, we also have
numerals for the numbers $-1$ through $-9$.  Let us give these the
numerals $\hat 1$ through $\hat 9$.  If you count up to $1,000,000$
and then count back down to 1, you will wind up writing $1$ as
$t^{6}1t^{-1}\hat 9t^{-1}\hat 9t^{-1}\hat 9t^{-1}\hat 9t^{-1}\hat
9t^{-1}\hat 9$. i.e., as $1 \hat 9 \hat 9 \hat 9 \hat 9 \hat 9 \hat
9$.  Evidently, we can write an arbitrarily long word for the number
$1$. 

We say that a \dehn algorithm is {\em finite to one} if as $x$ varies
over all words representing a fixed element of $G$, $R(x)$ takes only
finitely many values. 

\begin{remark}
{\rm
There are \dehn algorithms which are not finite to one 
off the identity.
}
\end{remark}

For the purposes of Section~\ref{geofin} we would 
like to modify our \dehn algorithm to avoid this behavior.

Given a  reduced word $w=t^{n}g_{n}t^{-1}g_{n-1}\dots g_{1}
t^{-1}g_{0}$, we call $n$ the {\em height} of $w$.
Choose a positive integer $N$ such that $(M/3)^N > K$. We add the
following additional rules to our system. 

\begin{enumerate} 
\item[(5)] If $w$ is a reduced word, as above, with height at
most $N$, such
that $\lg(w) < \frac{1}{2} \ell(g_0)$, replace $w$ by
an equivalent geodesic word in $\calg^*$.  
\end{enumerate}

Since there are only finitely many reduced words of height at most
$N$, this introduces only finitely many rules.

\begin{lemma} \label{heightbound}
Let $w$ be the reduction of a word in $\calg^*$ with respect to rules
1-5. If the height of $w$ is $n$, and $w$ does not represent the
identity, then $\lg(w) \geq (M/3)^n$. 
\end{lemma}

\begin{proof}
For height $n = 0$ the lemma is clear. For $n > 0$ we can write $w =
tw'\tinv g_0$, where $w'$ is reduced, of height $n-1$, and not the
identity, and $g_0$ is geodesic. 
For $n \leq N$, Rule 5 ensures that $\lg(w) \geq
\frac{1}{2} \ell(g_0)$. It follows that $3\lg(w) \geq \lg(w)+\ell(g_0) \geq
\lg(tw'\tinv) \geq M\lg(w')$. By induction the
lemma holds for all $n\leq N$. 

For $n > N$, writing $w$ as before, $\lg(w) \geq M\lg(w') -
\ell(g_0)$. By induction, $\lg(w') \geq (M/3)^{n-1}$. Also $\ell(g_0) \leq
2K$ which, by our choice of $N$, is less than $2(M/3)^{n-1}$. Therefore
$\lg(w) \geq (M - 2)(M/3)^{n-1} \geq (M/3)^n$ since $M > 3$. 
\end{proof}

\begin{corollary} \label{expendisvnil}
If $G$ admits an expanding endomorphism then $G$ is virtually
nilpotent.
\end{corollary}

\begin{proof}
Each element $g\in G$ can be represented by a word $w$ whose length
is bounded by $k\ln(\lg(g)+1)$, for some $k>0$.  Since there are
only polynomially many such words, $G$ has polynomial growth and
hence is virtually nilpotent.
\end{proof}

It is apparently unknown whether all torsion free nilpotent groups
have expanding endomorphisms. However, we will see in the next section
that they all have \dehn algorithms.

\begin{theorem}\label{expando}
If $G$ has an expanding endomorphism, then $G$ has a finite to one 
\dehn algorithm.
\end{theorem}

\begin{proof}
As in Corollary~\ref{expendisvnil}, the length $\lg(g)$ of an element
$g\in G$ gives a bound for the maximum length of any reduced normal
form representing $g$. Therefore there are at most finitely many possible
reduced normal forms for each element.
\end{proof}

\begin{remark}
{\rm
The results of this section remain valid under the weaker
hypothesis that $G$ has a finite index subgroup $H$ which admits an
expanding endomorphism $\phi$ with respect to $\lg$. The only change that needs
to be made is to replace $\phi(G)$ with $\phi(H)$ throughout. 
}
\end{remark}

This has the following corollary which we will need in our work on
geometrically finite groups. 

\begin{corollary} \Label{expandIsNGeodesic}
Let $G$ be finitely generated and suppose that $G$ has a finite index subgroup
which has an expanding endomorphism.  Let $\calg$ be a set of semi-group
generators for $G$.  Then for any $N>0$ there exists a \dehn algorithm as
above, with working alphabet $\A = \calg \cup \{t,t^{-1}\}$, such that any
normal form word $w$ with $\lg(w)<N$ is a geodesic word in $\calg^*$.  In
particular, this holds when $G$ is finitely generated and virtually abelian. 
\end{corollary} 

\begin{proof}
Let $H$ be a finite index subgroup with expanding endomorphism. (In the
virtually abelian case, $H$ is a finite index free abelian subgroup.)
Raising to a sufficient power furnishes us with an expanding endomorphism of
$H$, with expansion factor $M$ such that $M/3 > N$. By
Lemma~\ref{heightbound}, any normal form word $w$ with $\lg(w)<N$ has height
$0$.
\end{proof}

We will call such a \dehn algorithm {\em $N$-geodesic}.

We will say that a rule $u\to v$ is a {local geodesic} rule if both
$u$ and $v$ are words in the group generators and $v$ is a geodesic.
We will say that an $N$-geodesic \dehn algorithm $(\calg,\A,S)$ is
{\em $N$-tight} if $(\calg,\A,S\cup R)$ is also a $N$-geodesic \dehn
algorithm whenever $R$ is a finite set of local geodesic rules and
the left hand sides of $S$ and $R$ are disjoint.

We record here the following observation.

\begin{proposition} \Label{expandIsTight}
The $N$-geodesic \dehn algorithms of Corollary~\ref{expandIsNGeodesic}
are $N$-tight. 
\end{proposition}

\begin{proof}
In this case, the rules of $S$ determine that any sufficiently long
geodesic $g$ is replaced with a word $tg't^{-1}g''$ where $g$ and $g'$
are shorter geodesics.  On the other hand, $S$ replaces any
non-geodesic shorter than this with a geodesic.  In particular, no
rule of $R$ is ever applied.
\end{proof}

\section{Nilpotent Groups}\label{nilpt}

\def\gmt{{\gamma(t)}}
\def\fm{{f_\mu}}
\def\fmgt{{\fm(\gmt)}}
\def\fms{\fm_*}
\def\lm{\lim_{\Delta t \to 0}}
\def\vdt{v(\Delta t)}

\def\unr{U_n(\R)}
\def\unz{U_n(\Z)}
\def\ddt{\frac{d}{dt}}

\noindent 
In this section we shall show that every finitely generated, torsion
free nilpotent group embeds in a group which has an expanding endomorphism.  It
follows from Theorem~\ref{expando} and Theorem~\ref{subgroups} that
every torsion free nilpotent group has a \dehn algorithm. Since every
finitely generated nilpotent group is virtually torsion free, it
follows by Theorem~\ref{finiteindex} that every finitely generated
virtually nilpotent group has a \dehn algorithm.

We start with the group of $n\times n$ upper triangular matrices with
$1$'s on the diagonal.  Those with integer entries we denote by
$U_n(\Z)$, those with real entries we denote by $U_n(\R)$.

For each $\mu \in \R$ define $f_\mu:U_n(\R)\to U_n(\R)$ as follows.
If $x=(x_{ij})\in U_n(\R)$, set 
\[(f_\mu(x))_{ij} = \mu^{j-i}x_{ij}.\]
It is not hard to see that $f_\mu$ is a homomorphism and that if $\mu
\in \Z$ then $f_\mu:U_n(\Z)\to U_n(\Z)$.

Let us fix a generating set $\calg$ for $U_n(\Z)$ and endow $U_n(\R)$ with a
left invariant metric.

\begin{lemma}
The action of $\unz$ on $\unr$ is co-compact by isometries and fixed
point free.
\end{lemma}

\begin{proof}
We wish to see that each $x=(x_{ij}) \in \unr$, is a bounded distance
away from some $z = (z_{ij}) \in \unz$.  For $p = 1,\dots n-p$, let
$m(x_{1},\dots,x_{n-p})$ be the upper triangular matrix with 1's on
the diagonal, $x_1, \dots, x_{n-p}$ located distance $p$ above the
diagonal, and 0's everywhere else.  Notice that multiplying
$x=(x_{ij})$ by such an $m_p$ leaves unchanged the entries of $x$
below the $p^{\rm th}$ off-diagonal and adds $x_1, \dots, x_{n-p}$ to
the entries of $x$ on the $p^{\rm th}$ off-diagonal.  Consequently, we
can choose $m_1, m_2,\dots , m_{n-1}$ each with entries between 0 and
1 so that $z=x m_1 m_2 \dots m_{n-1} \in \unz$.  Since the entries of
each $m_i$ are bounded in size, so is their product.  Hence $z$ is a
bounded distance away from $x$ as required.
\end{proof}

Consequently,
\begin{lemma}
There is $\lambda=\lambda_\calg$ so that the embedding of the Cayley graph
$\Gamma_\calg(\unz)$ into $\unr$ is a $(\lambda,0)$ quasi-isometry.
\end{lemma}

\begin{proof} 
It is a standard result that a co-compact discrete
isometric action on a geodesic metric space induces a
$(\lambda,\epsilon)$-quasi-isometry.  It is not hard to see that in
the case of a fixed point free action, we may take $\epsilon=0$.
\end{proof}

\begin{lemma} 
For $\mu >1$, the map $f_\mu$ is a $\mu$-expanding endomorphism on $\unr$.
That is, for $x,y \in \unr$, $d(f_\mu(x),f_\mu(y)) \ge \mu d(x,y)$.  
\end{lemma}

\begin{proof}
It suffices to show that $f_\mu$ is everywhere infinitesimally
$\mu$-expanding. For $X\in\unr$, a tangent vector at $X$ is given by
$\ddt (X + At)$ where $A = (a_{ij})$ with $a_{ij} = 0$ for all $j\leq
i$. Without loss of generality we may assume $\|\ddt (I + At)\| =
(\sum a_{ij}^2)^{1/2} =: n(A)$. Then by left invariance and linearity
of matrix multiplication
\[\left\|\ddt (X + At)\right\| = n(X^{-1}A).\]
Clearly $f_\mu$ extends linearly to all upper-triangular matrices and we
have
\begin{align*}
\left\|\ddt f_\mu(X + At)\right\| &= \left\|\ddt (f_\mu(X) + f_\mu(A)t)\right\| \cr
        &= n(f_\mu(X^{-1}A)) \geq \mu n(X^{-1}A).
\end{align*}
\end{proof}

\begin{lemma}
For $\mu \in \Z$, $f_\mu(\unz)$ is finite index in $\unz$.
\end{lemma}

\begin{proof}
The proof is the same as the proof that $\unz$ is co-compact in
$\unr$. 
\end{proof}

Consequently,
\begin{lemma}
If $\mu > \lambda^2$, and $\mu \in \Z$ then $f_\mu$ is an expanding
endomorphism of $\unz$. \qed
\end{lemma}

Hence, by Theorem~\ref{expando},
\begin{lemma}
$\unz$ has a finite to one \dehn algorithm. \qed
\end{lemma}

Now it is a theorem (see \cite{Seg}, Chapter~5) that
\begin{theorem}
If $G$ is a finitely generated, torsion free nilpotent group then $G$
embeds in $\unz$, for some $n>0$. \qed
\end{theorem}

Hence, by Theorem~\ref{subgroups} and Theorem~\ref{finiteindex},
\begin{theorem}\label{virtuallyNilpotent}
If $G$ is finitely generated and virtually nilpotent, then $G$ has a
\dehn algorithm. \qed
\end{theorem}

\section{Relatively hyperbolic groups} \label{geofin}

In this section we prove a theorem concerning \dehn algorithms for (strongly)
relatively hyperbolic groups.  We first proved this in the context of
geometrically finite hyperbolic groups and these are the parade examaple of
relatively hyperbolic groups.  The statement and proof here are close
parallels of the geometrically finite case.

There are multiple equivalent definitions of what it means for a group to be
(strongly) hyperbolic relative to a collection of subgroups $\{\calp_1,\dots,
\calp_k\}$.  These are equivalent to Farb's \cite{F} definition of relative
hyperbolicity together with his bounded coset penetration property.  Usage of
the term relatively hyperbolic varies slightly in that it is often possible to
drop the requirement that the subgroups be finitely generated.  In our usage
these will all be finitely generated.

The key geometric result is the relation the geodesics and horoballs
of the negatively curved space to the geodesics of subspace upon which
the group acts co-comactly.  This is Lemma~\ref{quasiGeodesics} here,
the Morse lemma, Proposition~8.28 of \cite{DS}. 
%%, Proposition~YYY of \cite{O}.

\begin{theorem} \label{relativelyHyperbolicGroups}
Suppose that $G$ is hyperbolic relative to $\calp = \{P_1,\dots,
P_k\}$.  Suppose also that for each $i$, $1 \le i \le k$ and any $N$,
$P_i$ has a \dehn algorithm with is $N$-tight.  Then $G$ has a \dehn
algorithm.  This \dehn algorithm consists of local geodesic rules
together with \dehn algorithms for the $P_i$.
\end{theorem}

\begin{corollary} \label{geofingroup}
If $G$ is a geometrically finite hyperbolic group, then $G$ has a
\dehn algorithm.
\end{corollary}

\begin{corollary}\label{hypgphmflds}
If $M$ is a graph manifold each of whose pieces is hyperbolic
then $\pi_1(M)$ has a \dehn algorithm.
\end{corollary}

\begin{corollary}  \label{negativelyCurvedManifold}
Suppose that $M$ is a finite volume negatively curved manifold with
  curvature bounded below and bounded away from zero. Then $\pi_1(M)$
  has a \dehn algorithm.
\end{corollary}

\begin{corollary} \label{freeProductWithAmalgamation}
Suppose that $A$ and $B$ are groups with $N$-tight \dehn algorithms and that
$C$ is a finite group which includes as a subgroup of each of these.  Then
$A*_CB$ has a \dehn algorithm.
\end{corollary}

Corollaries \ref{geofingroup} and \ref{hypgphmflds} follow directly
from Theorem~\ref{relativelyHyperbolicGroups} since the groups in
question are hyperbolic relative to abelian (or virtually abelian)
groups.  Corollary~\ref{freeProductWithAmalgamation} follows since the
amalgam is hyperbolic relative to its factors. In the case of
Corollary~\ref{negativelyCurvedManifold} the groups are hyperbolic
relative to nilpotent groups \cite{F}.  Nilpotent groups have \dehn
algorithms by Theorem~\ref{virtuallyNilpotent}, but there is no
guarantee that these are $N$-tight for arbitrary $N$.  It is only in
the perhaps larger group of upper triangular matrices where this is
guaranteed. However, once we have proved
Theorem~\ref{relativelyHyperbolicGroups}, we will see how to proceed
here.

We suppose that $G$ is hyperbolic relative to a finite collection of
subgroups $\{P_1,\dots,P_k\}$.  The {\em parabolic} subgroups of $G$
are the $G$-conjugates of $\{P_1,\dots,P_k\}$.  We take $\calp$ to be
the set parabolic subgroups.    The following are well known properties
of relatively hyperbolic groups. See, for example, \cite{B}, \cite{O}
and \cite{DS}. 

\begin{review}\Label{review} 
\ 
\begin{enumerate}
\item $G$ acts discretely by isometries on a $\delta$-hyperbolic space
  $\H$.
\item This action induces an action on the boundary $\delH$.
\item There is a $G$ equivariant family of horoballs $\{B_P \mid P \in
  \cal P\}$.
\item For each $P \in \cal P$ we take $S_P$ to be $\partial B_P$.  $P$
  acts co-compactly on $S_P$.
\item $G$ acts co-compactly on $X = \H \setminus (\cup_{P\in\calp} B_P)$.
\item Each horoball $B_P$ is quasiconvex. Consequently, there is a
  rectraction $r_P: X \to S_P$ which is inherited from the hyperbolic
  retraction of $\H$ onto $S_P$.  (We will also refer to this
  retraction as $r_S$ where $S$ is the boundary of $P$.
\item \label{shrink} For points sufficiently distant from $S_P$, the retraction $r_P$
  shrinks $X$ distance by a super-linear factor.  That is to say,
  there is  a function $s(\cdot)$ with the property that for any linear
  function   $y(x)=mx +b$, there is $x_0$ such that for $x>x_0$, $s(x)
  > y(x)$  and there is $d_0$ so that if $d =
  \min(d_X(p,S_P),d_X(q,S_P)) > d_0$ then
  $$d_X(r_P(p),r_P(q)) < \frac{d_X(p,q)}{ s(d)}.$$
\item There is $\delta$ with the following property.  Suppose that
  $S_0$ and $S_1$ are disjoint horospheres, i.e., the boundaries of
  disjoint horoballs in $\calp$.  Suppose that $\gamma$ and $\gamma'$
  are $\H$ geodesics that start in $S_0$ and end in $S_1$ and that $x$
  and $x'$ are the last points of $\gamma$ and $\gamma'$ in $S_0$.
  Then $d_X(x,x') \le \delta$.
\item \label{boundedHorosphereRetraction} There is $\delta$ so that if
  $S_0$ and $S_1$ are disjoint 
  horospheres, then $r_{S_0}(S_1)$   has $d_X$ diameter bounded by
  $\delta$. 
\item \label{horoballGeodesicsDepartX} Given $\delta$ there is
$\epsilon$ with the following property.  Suppose $S$ is the boundary
of horoball $B$. If $\gamma$ is an $\H$ geodesic that starts and ends
on $S$ then the only portion of $\gamma$ lying in the $\delta$
neighborhood of $\H \setminus B$ are an initial and terminal segment
of $\gamma$, each of length at most $\epsilon$.
\end{enumerate}
\qed 
\end{review}

We need the following lemma which is Proposition~8.28 of \cite{DS}.

\begin{lemma}\Label{quasiGeodesics}
There is $\delta$ depending only on $\lambda$ and $\epsilon$ with the
following property.  Suppose that $w$ is a $(\lambda,\epsilon)$
quasigeodesic in $X$ and $\gamma$ is a $\H$ geodesic with the same
endpoints. Suppose that $\Sigma = \Sigma(\gamma)$ is the union of
$\gamma$ and the horospheres that it meets.  Then $w$ lies in a
$\delta$ neighborhood of $\Sigma$. \qed
\end{lemma}

Given an $\H$ geodesic, $\gamma$, it meets a finite (possibly empty)
collection of horoballs, $B_i,\dots,B_k$.  Replace each portion
$\gamma \cap B_i$ with an $X$ geodesic, $\sigma_i$ to produce the $X$
paths
$$\sigma = \gamma_0\sigma_1\gamma_1\dots\sigma_n\gamma_n.$$
We refer to a path formed in this way as a {\em rough geodesic}.

\begin{lemma}\Label{roughGeodesics} 
There is a $\lambda$ such that every rough geodesic is an $X$
$\lambda$ quasigeodesic.
\end{lemma}

\begin{proof}
Suppose that $\gamma$ is an $H$ geodesic and $\sigma =
\gamma_0\sigma_1\gamma_1\dots\sigma_n\gamma_n$ is a corresponding
rough geodesic.  Suppose that $\sigma'$ is a corresponding $X$
geodesic.  By Lemma~\ref{quasiGeodesics}, this lies in a $\delta$
neighborhood of $\Sigma = \Sigma(\gamma)$.  Let us decompose $\sigma'$
  as   $\sigma' =
\gamma_0'\sigma_1'\gamma_1'\dots\sigma_n'\gamma_n'$ where $\sigma_i'$
is the portion of $\sigma'$ which lies within $\delta$ of the
horosphere for $\sigma_i$, but not within $\delta$ of $\gamma$.  Some
of these may be empty.  However, it follows that for each $i$,
$\gamma_i$ and $\gamma_i'$ lie within $2\delta$ of each other.  Since
each of these is geodesic, the difference in their lengths is
bounded.  Similarly, for each $i$, the endpoints of $\sigma_i$ and
$\sigma_i'$ are close to each other, thus bounding the difference in
their lengths.  Accordingly, the difference in lengths along $\sigma$
and $\sigma'$ arise only from these breakpoints each of which
contributes only a bounded difference.  Since there is a mimimum
distance between horospheres, these breakpoints are bounded away from
each other.  The result follows.
\end{proof}

We record here two general properties of $\delta$-hyperbolic
spaces. (Here we use the parameterized version of
$\delta$-hyperbolicity.) 

\begin{proposition}\Label{goodBends}
Given $\delta'>\delta$, there are $(\lambda,\epsilon)$ with the
following property.  Suppose $\gamma$ is a piecewise geodesic.
Suppose that each segment of $\gamma$ has length at least $\delta'+1$,
and that at each bend, both segments depart a $\delta$ neighborhood of
each other after travelling at most distance $\delta'$ from that
bend.  Then $\gamma$ is a $(\lambda,\epsilon)$ quasigeodesic. \qed
\end{proposition}

\begin{proposition}\Label{quasiQuasiQuasi}
Suppose that $(\lambda_1,\epsilon_1)$ and $(\lambda_2,\epsilon_2)$ are
given.  Then there is $(\lambda_3,\epsilon_3)$ with the following
property.  If $\sigma$ is a $(\lambda_1,\epsilon_1)$ quasigeodesic and
$\tau$ is formed from $\sigma$ by replacing disjoint subpaths with
$(\lambda_2,\epsilon_2)$ quasigeodesics, the $\tau$ is a
$(\lambda_3,\epsilon_3)$ quasigeodesic. \qed
\end{proposition}

Suppose $G$ is hyperbolic relative to $P_1,\dots,P_k$. We would like
to find a generating set $\calg$ in which $P_1,\dots,P_k$ are convex
in the Cayley graph of $G$.  Given an set of generators $\calg'$ for
$G$ and $K>0$, set
$$\cala_i(K) = \{g \in P_i \mid d_X(1,g) \le K\}$$
and
$$\calg = \calg(K) = \calg' \cup \left(\bigcup_{i=1}^k \cala_i(K)\right).$$

\begin{lemma} \Label{generators}
Given $K$ sufficiently large, $\calg$ has the following properties:
\begin{enumerate}
\item
There are constants $A$ and $B$ with the following properties: Suppose
$w$ is a $\calg$-geodesic.  Let $S_P$ be a horosphere with $P$
conjugate to $P_i$. Suppose $w$ begins and ends at $X$ distance at
most  $d$ from $S_P$.  Then
$w = xyz$, where $\ell(x) \le A d +B$, $\ell(z) \le A d +B$, and $y\in
(\cala_{i}(K))^*$.
\item
If $w$ begins on $S_P$, $x$ is empty.  If $w$ ends on $S_P$, $z$ is empty.  In
particular, a $\calg$-geodesic evaluating into $P_i$, is written in letters
all of which lie in $P_i$.
\item
If we fix $d$ then if $w$ is sufficiently long, $y$ is
non-empty. 
\end{enumerate}
\end{lemma}

\begin{proof}
We claim that there is a bound $r$ independent of $K$ so that if $e$
is a $\calg(K)$ edge which does not lie in $P$, then the the $X$
length of $r_P(e)$ is less than $r$.  If $e$ is an $\cala_i(K)$ edge
which does not lie in $P$, then it lies within a bounded distance of
some horosphere other than $S_P$.  By
property~\ref{boundedHorosphereRetraction} of
Proposition~\ref{review}, $r_{S_P}(S_{P'})$ has bounded diameter.
There are only finitely many $\calg'$ letters and their edges also have
bounded retractions onto $S_P$.  This gives the bound $r$.  

Now consider the case of a geodesic $w$ which begins and ends in $P$.
We wish to show that all edges of $w$ lie in $P$.  If this fails, we
replace $w$ with a sub-segment whose only contact with $P$ are its
two endpoints, $p$ and $q$.  Notice that it must therefore have length
at least 2 since it leaves and returns to $S_P$.  Let $d_1$ be the
maximum distance from $S_P$ to $P$.  Then the path $r_P(w)$ starts and
ends within distance $d_1$ of $w$.  Thus $\ell(w) \ge \frac
{d_X(p,q)-2d_1}{r}$.  Now consider an $X$-geodesic from $p$ to $q$.
This has length at most $d_X(p,q)+2d_1$ and each point of it lies
within distance $d_1$ of $P$.  It follows that the $\cala_i(K)$
distance between $p$ and $q$ is at most $\frac{d_X(p,q)+2d_1}{K-d_1} +
1$.  Choosing $K$ sufficiently large contradicts the assumption that
$w$ was geodesic.

Now consider the case in which $p$ and $q$ do not necessarily lie on
$S_P$. Let $p''$ and $q''$ be their respective projections onto $S_P$
and $p'$ and $q'$ be points of $P$ near these.  There are $\lambda$
and $\epsilon$ depending on $K$ so that the embedding of $G$ into $X$
is a $(\lambda,\epsilon)$ quasi-isometry.  Consider $tuv$ with $t$ a
geodesic from $p$ to $p'$, $u$ a geodesic from $p'$ to $q'$ and $v$ a
geodesic from $q'$ to $q$.  Then
 $$\ell(w) \le \ell(tuv) \le 2\lambda d + 2\epsilon +
 \frac{d_X(p',q')+2d_1}{K-d_1} + 1.$$ 
Now if $y$ does not appear in $w$, i.e., $w$ contains no subword lying
 in $P$, then
$$\frac{d_X(p',q')-2d_1}{r} + 1 \le \ell(w)\le 2\lambda d + 2\epsilon +
 \frac{d_X(p',q')-2d_1}{K+d_1} + 1.$$
The value of $\lambda$ can only decrease as $K$ increases, since
 $\lambda$ measures how many $\calg(K)$ letters it takes to travel a
 certain distance in $X$, and for $K$ sufficiently large, $K-d_1 >
 r$.  Thus, for any sufficiently large $K$, there is a linear bound
 $\ell(w) < A'd +B'$ on those $w$ for which $y$ is empty.

We now suppose $w = xyz$ where $y$ is the maximal portion of $w$ lying
in $P$ and is non-empty.  Let $p'''$ and $q'''$ be the endpoints of
$y$.  We claim that these must lie a bounded distance from $p'$ and
$q'$.  To see this, notice that $w$ is an $X$ quasi-geodesic.  It
follows from Lemma~\ref{quasiGeodesics} that $w$ fellow travels the
its $X$ geodesic union the horosphere's that these meet.  It is not
hard to see that if $d_X(p,q)$ is sufficiently large, this $X$
geodesic meets $S_P$ near $p''$ and $q''$.
\end{proof}

For $i \ne j$, $P_i$ and $P_j$ meet in a finite (perhaps trivial) subgroup.
We will assume that $K$ is chosen large enough so that any non-trivial
elements common to one or more subgroups appear as generators. After choosing
$K$, we will refer to $\calg(K)$ and $\cala_i(K)$ as $\calg$ and $\cala_i$.

We are now in a position to describe the \dehn algorithm of
Theorem~\ref{relativelyHyperbolicGroups}.  This depends on constants $D$ and
$E$.  For each $i$, let $(\cala_i,\A_i,S_i)$ be a $D$-tight \dehn algorithm
for $P_i$. We will assume that any rules operating inside a common subgroup
rewrite immediately to a single letter and thus, these rules agree between the
different $S_i$. We take $S_\calg$ to be a collection of local geodesic rules
which contain a left-hand side for each $\calg$ word which is not a
geodesic. We assume that these agree with any rules which also appear in some
$S_i$.  We will assume $D \ge E$.  We take $\A = \calg \cup (\cup_i \A_i)$ and
$S = S_\calg \cup (\cup_i S_i)$.  We will show that with $E$ sufficiently
large, $\cald = (\calg,\A,S)$ is a \dehn algorithm.  This requires a series of
lemmas.

We first check that the parabolic subgroup sub-\dehn algorithms are
still effectively $D$-tight within $\cald$.

\begin{lemma}\Label{goodParabolicSubwords} 
Suppose that $w$ is the result of $\cald$ reducing a $\calg$ input
word and that $u$ is a maximal $\A_i$ subword of $w$.  Then $u$ is a
reduced word for a $D$-geodesic \dehn algorithm for $P_i$.
\end{lemma}

\begin{proof}
Consider the process by which $u$ is produced.  Since the $\A_i$ are
disjoint, for $j\ne i$, no $S_j$ rule can apply in the production of
$u$.  Consequently the formation of $u$ is carried out by $S_i$ rules
and $S_\calg$ rules.  Notice that any non-$\cala_i$ input letters
which are consumed in the production of $u$ must first be turned into
$\cala_i$ letters prior to their consumption by $S_i$.  This is done
by $S_\calg$ rules shortening non-geodesics into geodesics which must
be in $\cala_i$ letters.  Therefore, $u$ could have been produced by
applying the $S_\calg$ rules and $S_i$ rules to an input word in
$\cala_i^*$.  The result now follows from the assuption that \dehn
algorithm for $P_i$ is tight.
\end{proof}

It now follows that if $w$ is the result of $\cald$-reducing a
$\calg^*$ input word, then $w$ consists of reduced words from the
parabolic subgroups alternating with $E$-local geodesics which do not
contain any parabolic letters.  These $E$-local geodesics may be
empty, but by assuming that the parabolic subwords are maximal, we may
assume that no two adjacent parabolic subwords lie in the same
parabolic subgroup.  Note that if two or more $P_i$ meet in a
non-trivial finite subgroup, any ambiguity where one parabolic
subgroup ends and another begins can only consist of a single letter.

We will choose to decompose $w$ in a slightly different manner.  We
choose a parameter $F<D$.  We decompose $w$ as
$$w = g_0p_1,\dots,g_{m-1}p_mg_m$$
where the $p_j$ are the maximal parabolic subwords which represent
group elements of length greater than $F$.  Since all other maximal
parabolic subwords represent group elements of length less than or
equal to $D$, each is an $\cala_i$-geodesic for some $i$.  It follows
that the $g_j$ are $E$-local geodesics.  Again, some of the $g_j$ may
be empty, but not if they lie between $P_i$ words.

\begin{lemma} \Label{hyperbolicQuasiGeodesics}
For $D$, $E$ sufficiently large,there is $(\lambda,\epsilon) =
(\lambda_{D,E,F}, \epsilon_{D,E,F})$ such that each $g_i$ is a
$(\lambda,\epsilon)$-quasi-geodesic in $\H$.  While increasing $F$
weakens the quasi-geodesity, increasing $D$ and $E$ does not.
\end{lemma}

\begin{proof}
Let $v$ be an $E$-local geodesic of length $E$.  This is a Cayley
graph geodesic, and hence an $X$ $(\lambda,0)$-quasi geodesic, with
$\lambda$ depending only on the embedding of $\Gamma$ into $X$.  By
Proposition \ref{quasiGeodesics}, $u$ asynchronously fellow-travels
its $\H$ geodesic $\gamma$ together with any horospheres that $\gamma$
enters.  Now $\gamma$ cannot stray far into any horosphere, for
otherwise $u$ would contain parabolic subwords of length greater than
$F$.  This bounds the ratio between the $X$ length and the $\H$ length
of $\gamma$.  Notice that this bound is independent of $E$.  Thus, by
increasing $E$, we proportionally increase the $X$-length of $u$.
That is to say, there is $(\lambda',\epsilon')$ is a Cayley graph
geodesic containing no parabolic subword of length greater than $F$,
then $u$ is an $\H$- $(\lambda',\epsilon')$-quasigeodesic.

It is a standard result for $\delta$-hyperbolic spaces that given
$(\lambda',\epsilon')$, for $E$ sufficiently large, there is
$(\lambda,\epsilon)$ so that every $E$-local
$(\lambda',\epsilon')$-quasigeodesic is a $(\lambda,\epsilon)$
quasigeodesic.  Thus, choosing $E$ (and hence, $D$) sufficiently large
makes each $g_i$ an $\H$ $(\lambda,\epsilon)$ quasigeodesic as
required.
\end{proof}

Condsider the decomposition of $w$ into
$$w = g_0p_1,\dots,g_{m-1}p_mg_m$$
as above.  Ultimately, we must show that $w$ is empty if and only if
the input word which created it represents the identity.  We will
examine several paths related to $w$, namely
$$\sigma = \sigma(w,F) =
\gamma_0\pi_1,\dots,\gamma_{m-1}\pi_m\gamma_m$$
$$\pi = \pi(w,F) =
g_0\pi_1,\dots,g_{m-1}\pi_mg_m$$
$$\nu = \nu(w,F) =
g_0q_1,\dots,g_{m-1}q_mg_m$$
where
\begin{itemize}
\item
Each $\gamma_i$ is the $\HH$-geodesic for the corresponding $g_i$,
\item
Each $\pi_i$ is the $\HH$-geodesic for the corresponding $p_i$,
\item
Each $q_i$ is a Cayley graph geodesic for the corresponding $p_i$.
\end{itemize}

\begin{lemma} \Label{quasiGeodesicPaths}
Given $D$, $E$, $F$ sufficiently large, 
\begin{itemize}
\item
There is $(\lambda,\epsilon)$
such that $\sigma$ is an $\HH$- $(\lambda,\epsilon)$-quasigeodesic.
Increasing $D$ and $E$ does not worsen this quasigeodesity.
\item
There is $(\lambda,\epsilon)$
such that $\pi$ is an $\HH$- $(\lambda,\epsilon)$-quasigeodesic. 
\item
There is $(\lambda,\epsilon)$ such that $\nu$ is  a Cayley   graph
$(\lambda,\epsilon)$-quasigeodesic. 
\end{itemize}
\end{lemma}

\begin{proof}
We first consider $\sigma$.  We choose $F$ sufficiently large.  Since
each $\pi_i$ is long, by property~\ref{horoballGeodesicsDepartX} of
Proposition~\ref{review}, it spends only a limited time in a
neighborhood of the exterior of its horoball.  On the other hand, each
$\gamma_i$ can only spend a bounded time in the neighborhood of the
horoballs it starts and ends at, for otherwise, by Lemma~
\ref{generators}, it would start or end in the corresponding parabolic
letters, contradicting the maximality of $p_{i-1}$ (at its beginning)
or $p_i$ (at its end).  Thus, the only way, $\sigma$ can fail to
satisfy the assumptions of Proposition~\ref{goodBends} is if one or
more of the $\gamma_i$ is short, i.e., of $X$ length less than
$\delta'+1$.  In this case, we modify $\sigma$ to produce $\sigma'$ by
deleting each short $\gamma_i$ and replacing $\pi_i$ with $\pi'_i$
starting at the beginning of $\gamma_i$.  Clearly $\sigma$ and
$\sigma'$ asynchronously fellow travel.  By
Proposition~\ref{goodBends}, $\sigma'$ is an $\HH$ quasigeodesic, and
thus, so is $\sigma$.

It now follows by Lemma~\ref{quasiQuasiQuasi} that $\pi$ is an
$\HH$ quasigeodesic.

Finally, it follows from Lemma~\ref{roughGeodesics} that $\nu$ is an
$X$ quasigeodesic and hence a Cayley graph qusigeodesic.
\end{proof}

In the case where each $P_i = \langle\cala_i\rangle$ has the
falsification by fellow traveler property, this gives Lemma~4.7 of
\cite{NS}.  It then follows that the language of geodesics in $G =
\langle \calg \rangle$ is a regular language and that the growth of $G
= \langle \calg \rangle$ is rational.  This includes the limit groups
of \cite{Sel} since, as \cite{Da} has shown, these are hyperbolic
relative to abelian subgroups.

\begin{proof} (Theorem \ref{relativelyHyperbolicGroups}) 
We suppose that $w$ is the result of $\cald$-reducing an input word in
$v \in\calg^*$.  We must show that $w$ is empty if and only if $v$
represents the identity.  Since $w$ remembers its group element, the
``only if'' part is clear.

Suppose now that $v$ represents the identity.  Then $\sigma(w)$ is an $\HH$-
quasigeodesic.  Since it represents the identity, this bounds its length.
This, in turn bounds the length of $w$.  Recall that increasing $D$ and $E$
does not worsen the quasi-geodesity of $\sigma$, and thus does not degrade the
bound on the length of $w$.  We may then assume that $D$ and $E$ are greater
than this bound.  Thus, $w$ is a geodesic, in particular, a geodesic for the
identity, and thus empty as required.
\end{proof}

\begin{proof} (Corollary \ref{negativelyCurvedManifold})
Let $G = \pi_1(M)$ where $M$ is a finite volume negatively curved
manifold with curvature bounded below and bounded away from 0.  By
\cite{F}, $G$ is hyperbolic relative to nilpotent subgroups
$P_1,\dots,P_k$. Now each $P_i$ has a \dehn algorithm by
Theorem~\ref{virtuallyNilpotent}.  However, there is no guarantee that
this is $N$-tight for $P_i$.  It is, however, $N$-tight for matrix
group $\unz$.  Given any finite generating set $\calp$ for $P_i$, we
may include these into a generating set for $\unz$.  Now, if $N\ge 1$,
any $N$-tight \dehn algorithm for $\unz$ is $1$-tight.  It follows
that for each $p\in\calp$, $p$ is the unique reduced word for
itself. In particular, this is a $1$-tight \dehn algorithm for $P_i$. 

Since Lemma~\ref{generators} holds for any sufficiently large $K$, we
can assume that $\cala_i$ contains any finite subset of $P_i$ we
select.

Now consider the paths of Lemma~\ref{quasiGeodesicPaths}.  The
decompositions depend on a parameter, $F$, and this parameter is
stated in terms of Cayley graph length.  However, it is only used to
ensure that each $\pi_j$ is long, i.e., that the $\HH$ geodesic of
this group element is long.  By choice of $K$ and hence, $\cala_i$, we
can force this to be the case for any parabolic group element whose
reduced word is at least two letters long.  The proof now proceeds as
before. 
\end{proof}

\section{Histories, Compression, Splitting and Splicing}\label{splicingEtc}

\noindent This and the following section are devoted to showing that
certain groups do not have \dehn algorithms. In this section we
develop tools that apply to any deterministic length-reducing
rewriting system.  Thus we will be able to show that a particular
group $G$ has neither a \dehn algorithm, nor a non-incremental \dehn
algorithm.  We believe that these results also hold for
non-deterministic \dehn algorithms.  These latter are related to
growing context sensitive languages.  Extension of our methods to this
case is work in progress.

Let $w_0,\ldots,w_n$ be the sequence of words produced as a rewriting
algorithm makes $n$ substitutions on $w_0$. We call this sequence the
{\em history to time $n$ of $w_0$}.
We can draw a {\em diagram} of the history as follows. 
Draw $w_0$ as a row of $\ell(w_0)$ adjacent unit squares, labelled
with the letters of $w_0$. 
For each $i>0$ we draw $w_i$ below $w_{i-1}$ as follows. Draw a line
segment under the first \lhs\ appearing in $w_{i-1}$. (We call this a
{\em substitution line}.) Underneath it put a row of equal width,
height 1 rectangles, labelled with the corresponding \rhs, or if the
\rhs\ is empty, put a single black rectangle. Fill the remainder of the row
with a copy of whatever appears in that part of $w_{i-1}$. 

The {\em width of a letter of} $w_i$ is the width
of its rectangle in the diagram. The {\em width of a subword of}
$w_i$, not to be confused with its length, is the sum of the widths
of the letters making it up (i.e.\ disregarding any black rectangles). 

In order to get a handle on how the number of letters in a word
decreases as the algorithm runs, we consider how the widths of letters
increase. 

\begin{lemma} \label{widerletter}
Let $W$ be the length of the longest \lhs\ of the rewriting system. 
In a diagram, the letters of any \rhs\ (under a substitution line)
have width at least $W/(W-1)$ times that of the narrowest letter in
the corresponding \lhs\ (above the substitution line). 
\end{lemma}

\begin{proof}
If we were to first make all the letters of the \lhs\ equal in width,
deleting any black rectangles which appear, we would certainly not
make the narrowest letter any narrower. Then at least one letter, of at most
$W$, is removed, giving a further expansion of at least the
stated factor. 
\end{proof}

Next we define the {\em generation} of each letter in a diagram.
The generation of each letter in the first row is 0. The generation of
a letter in row $i>0$ is the generation of the letter above it, if it
is not in a \rhs, or one more than the least generation of the letters
above the substitution line, if it is in a \rhs. 

\begin{lemma} \label{width}
If the generation of a letter is $n$ then its width is at least
$\left(\frac{W}{W-1}\right)^n$. 
\end{lemma}

\begin{proof}
True for row 0. Suppose it is true for row $i-1$. Since each letter in
row $i$ not in a \rhs\ has the same generation and width as the
corresponding letter in the row above, the assertion holds for these
letters. By Lemma~\ref{widerletter} any letter in a \rhs\ is at least
$W/(W-1)$ times the width of the narrowest letter in the corresponding
\lhs. But the generation of each \rhs\ letter
exceeds the generation of the narrowest \lhs\ letter by {\em at most}
one. Since the assertion is assumed to hold for the narrowest letter in
the \lhs., it holds for the letters of the \rhs.
\end{proof}

% What we show next, roughly speaking, is that the amount of information
% an \ira\  carries from one half of a word to the other is bounded by
% the number of trips it makes between the two halves. 

\begin{figure}
\includegraphics[height=6cm]{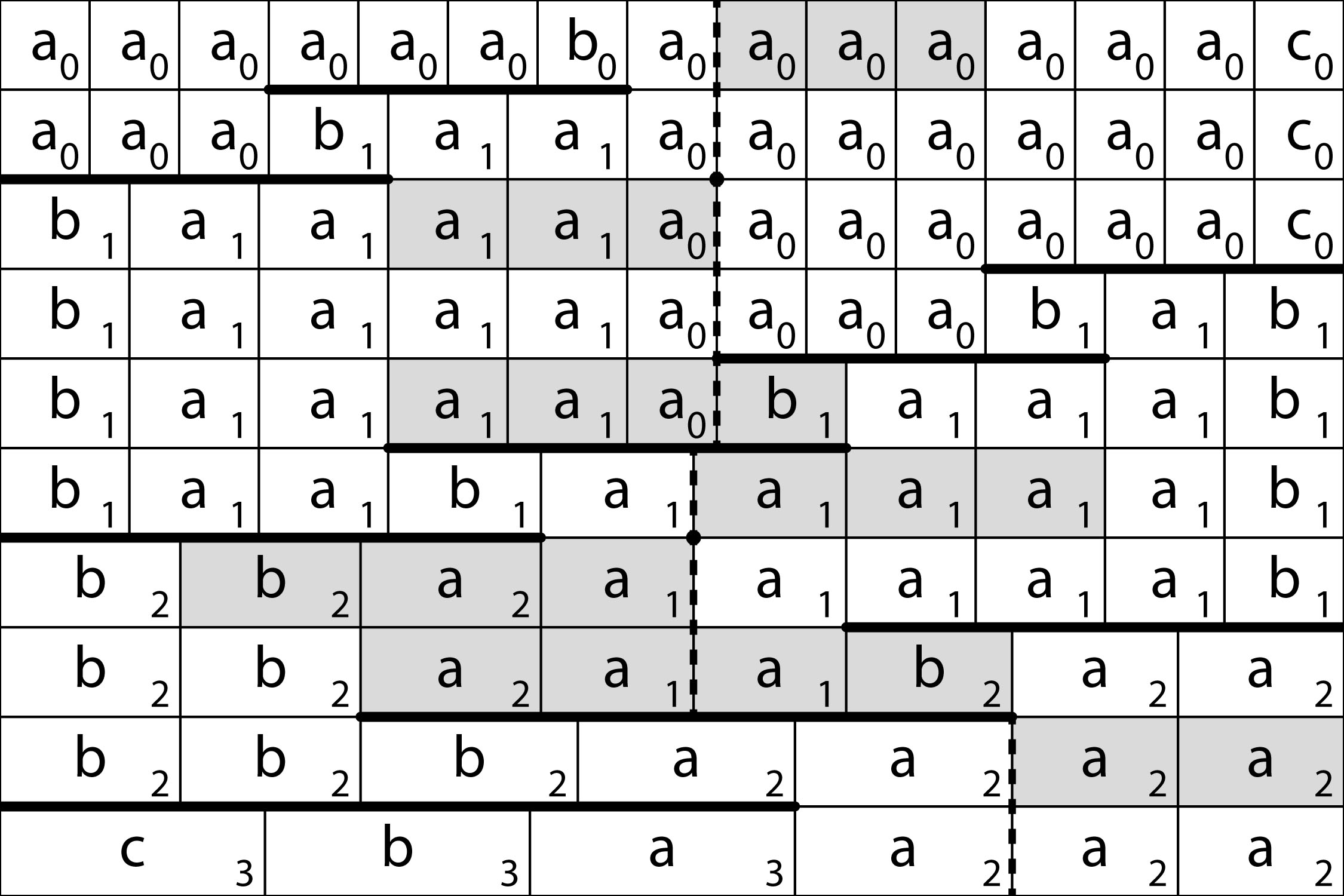}
\caption{A diagram for one possible history of a rewriting system. The
numbers indicate generations. Dotted lines indicate a
splitting path.}
\label{splittingpath}
\end{figure}

\begin{definition} 
See Figure~\ref{splittingpath}.  A {\em splitting path} of length $n$
in a diagram for $w_0\ldots w_t$ consists of $n$ vertical line
segments running between letters, from the top of the diagram to the
bottom, such that successive segments either join end to end, or are
linked by a substitution line. Segments may not cut substitution
lines. For each segment substitution lines 
between the top and bottom of the segment, all lie to the same side of
the segment. 
% Thus each word is split into two parts which we
% may call $w_t^-$ (to the left) and $w_t^+$ (to the right). 
\end{definition}

\begin{lemma} \label{pathexists}
If $w_t$ contains a letter of generation $g$, then the diagram 
contains a splitting path of length at most $2g+2$ ending next to the
letter. We may choose the path to end on either side of it.
\end{lemma}

\begin{proof}
Start at the bottom of the diagram with a vertical segment next to the
letter of generation $g$. Extend upward until we come to a
substitution line. Above that line will be a letter of generation
$g-1$. Start a new segment next to that letter and continue on
up. After hitting at most $g$ substitution lines we reach the top of
the diagram. (If we hit an endpoint of a substitution line
we start a new segment only if the letter we are following is
under the line.)

This is not yet a splitting path: our vertical segments could still
have substitution lines on both sides. When this happens it can only
be with substitutions to the left in the upper part of the segment and
to the right in the lower. (A sequence of substitutions going right to
left would have to cross the vertical segment because such
substitutions always overlap.) We split each such segment at the
appropriate point and we are done.
\end{proof}

Associated with each splitting path are its {\em details:} For each
vertical segment we record whether any substitutions take place to the
left or the right. (If neither, we can arbitrarily designate it as
left.) For a left segment we record the first $W-1$ letters to its
right (which will be constant), or to the end of the word if nearer. 
For a right segment we record the $W-1$ letters to the left, or
to the start of the word if nearer. If a segment ends on a
substitution line we record the \lhs, the position at which the path
splits it (in the range $0-W$) and the position at which the next
segment splits the \rhs\ (in the range $0-(W-1)$). 

We say that two splitting paths (in different diagrams) are
{\em equivalent} if they have the same details. 
(Note: we do not require vertical segments to be the same
height.) 

For example, the details of the splitting path shown in
Figure~\ref{splittingpath} might be given as: (left, ``aaa''), (right,
``aaa'', ``aaab'', 3, 2), (left, ``aaa''), (right, ``baa'', ``aaab'',
2, 3), (left, ``aa''). 

\begin{remark} \label{numsplits}
{\rm 
There are no more than $(2(W+1)^2|\A|^{2W+1})^{n+1}$ equivalence 
classes of splitting path of length less than or equal to $n$.
}
\end{remark}

Given a splitting path for $w_0,\ldots,w_t$, we define $w_i^-$ and
$w_i^+$ to be the subwords of $w_i$, to the left and the right
respectively of the path. 
% If $w_i$ contains a word on the
% path we split it at the position of the {\em outgoing} vertical
% segment, the one going from $w_i$ to $w_{i+1}$.
The next lemma can be interpreted as telling us that the detail of the
splitting path is like a message that is passed between
$w_0^-$ and $w_0^+$: if $v_0^-$ sends the same message as $w_0^-$,
$w_0^+$ won't notice the change.

\begin{lemma} \label{splicing}
Let $v_0,\ldots,v_r$ and $w_0,\ldots,w_s$ contain equivalent splitting paths. 
Then the history of $v_0^- w_0^+$, up to a suitable time, contains
an equivalent splitting path, and ends with the word $v_r^- w_s^+$.
\end{lemma}

\begin{proof}
Cut the histories of $v_0$ and
$w_0$ along their respective splitting paths. Fit the left
half of $v_0$'s history with the right half of $w_0$'s history. The
lengths of vertical segments are most likely unequal: one side or the
other is constant so we just make as many copies of the constant side
as required to fit the two together. 

We claim that
in the resulting sequence of words, each word differs from the
next by replacing a \lhs\ with its corresponding \rhs. 
This is clear when both words lie on the same segment or on successive
segments joined end-to-end. 
In the remaining case, both path details record the same \lhs, split at
the same point, and identical splitting points in the corresponding \rhs.

% The agreement of the
% words in the two splitting paths serves to ensure that the two sides
% match where the vertical segments end. 

We still have to show that the left-hand sides at which changes
occur are those that would be chosen by the algorithm. 
Consider words joined at a left segment. The \lhs\ begins in
$v_i^-$ (and ends in it as well, unless it is one of the \lhs s on the
path). Therefore, from the start of the \lhs, the next $W$ letters are
the same whether $v_i^-$ is completed by $v_i^+$ or $w_j^+$ (since
these begin with the same $W-1$ letters). The algorithm will therefore
substitute at the same place in either word. 
Now consider words joined at a right segment. The \lhs\ ends somewhere
in $w_j^+$. Any \lhs\ in $v_i^- w_j^+$, starting to the
left of this one, would have to start within $W-1$ letters of the end of
$v_i^-$, for otherwise it would be a \lhs\ in $v_i$ (wholly to the
left of a right segment). But in view of this, the same \lhs\ would appear
in $w_j$ (since $w_j^-$ and $v_i^-$ end with the same $W-1$
letters). 

It follows that we have constructed the history of $v_0^- w_0^+$.
That it contains a copy of the same splitting path, and ends with
$v_r^- w_s^+$, is clear.
\end{proof}

We would like to be able to say that $v_r^-$ is determined by $v_0^-$
and the splitting path but unfortunately this is not quite true. Let
$[v_r^-]$ denote the first word to the left of the last
segment in the splitting path. What the proof of
Lemma~\ref{splicing} shows is that $[v_r^-]$ is determined by $v_0^-$
and the splitting path. If the last segment is a right segment,
then $v_r^- = [v_r^-]$, but if not
the best we can say is that $v_r^-$ is obtained from
$[v_r^-]$ by substitutions entirely inside the latter. Similar
statements hold for $w_s^+$. 

\subsection{Subwords and border letters}
We extend the results of this section to subwords. If
$v_0,\ldots,v_t$ is a history of $v_0$, and $w_0$ is a subword of
$v_0$, how shall we define the history of $w_0$? We can do it by
fixing a {\em deletion convention} for the rewriting system: for each
\lhs, decide which letters are deleted and which are changed to get
the corresponding \rhs. It is then determined, when a substitution takes
place over the boundary of $w_i\subset v_i$, which letters of the
\rhs\ belong to $w_{i+1}$ and which do not. More generally we can
consider $v_0$ to be split up into arbitrarily many subwords; a
deletion convention will determine how each $v_i$ is to be split up. 

We want to define the diagram of $v_0$'s history in such a way that
each subword gets its own ``sub-diagram''. In other words, we want the
history of $w_0\subseteq v_0$ to occupy a rectangular block underneath
$w_0$. Therefore, when a substitution takes place over a subword
boundary, we adjust the widths of the \rhs\ letters on either side of
the new boundary to keep it vertically aligned under the previous
boundary. The problem that arises is that a deletion may occur on one
side only of the subword boundary: in that case
Lemma~\ref{widerletter} fails. We make the following adjustments. 

%We could define the diagram of $w_0 \subseteq v_0$ exactly as
%before (excepting that now not every row need contain a substitution),
%but we run into the following difficulty. For substitutions over the
%boundary of $w_i$ such that all deletions take place outside of $w_i$,
%Lemma~\ref{widerletter} no longer holds. Therefore we make the
%following adjustments. 

We designate the $W-1$ letters to either side of a subword boundary as
{\em border letters} (see Figure~\ref{subword}).  When a \rhs\
contains both border and non-border letters, we assign widths as
follows. The number of border letters will be the same in the \rhs\ as
in the \lhs, so we line them up under the border letters of the \lhs\
and keep their widths the same; we expand the non-border letters to
fill the remaining space evenly. Otherwise we assign widths as
previously stated. Note that when a \lhs\ contains a subword boundary,
the \rhs\ will consist only of border letters. Now
Lemma~\ref{widerletter} holds for all non-border letters.

\begin{figure}
\includegraphics[height=6cm]{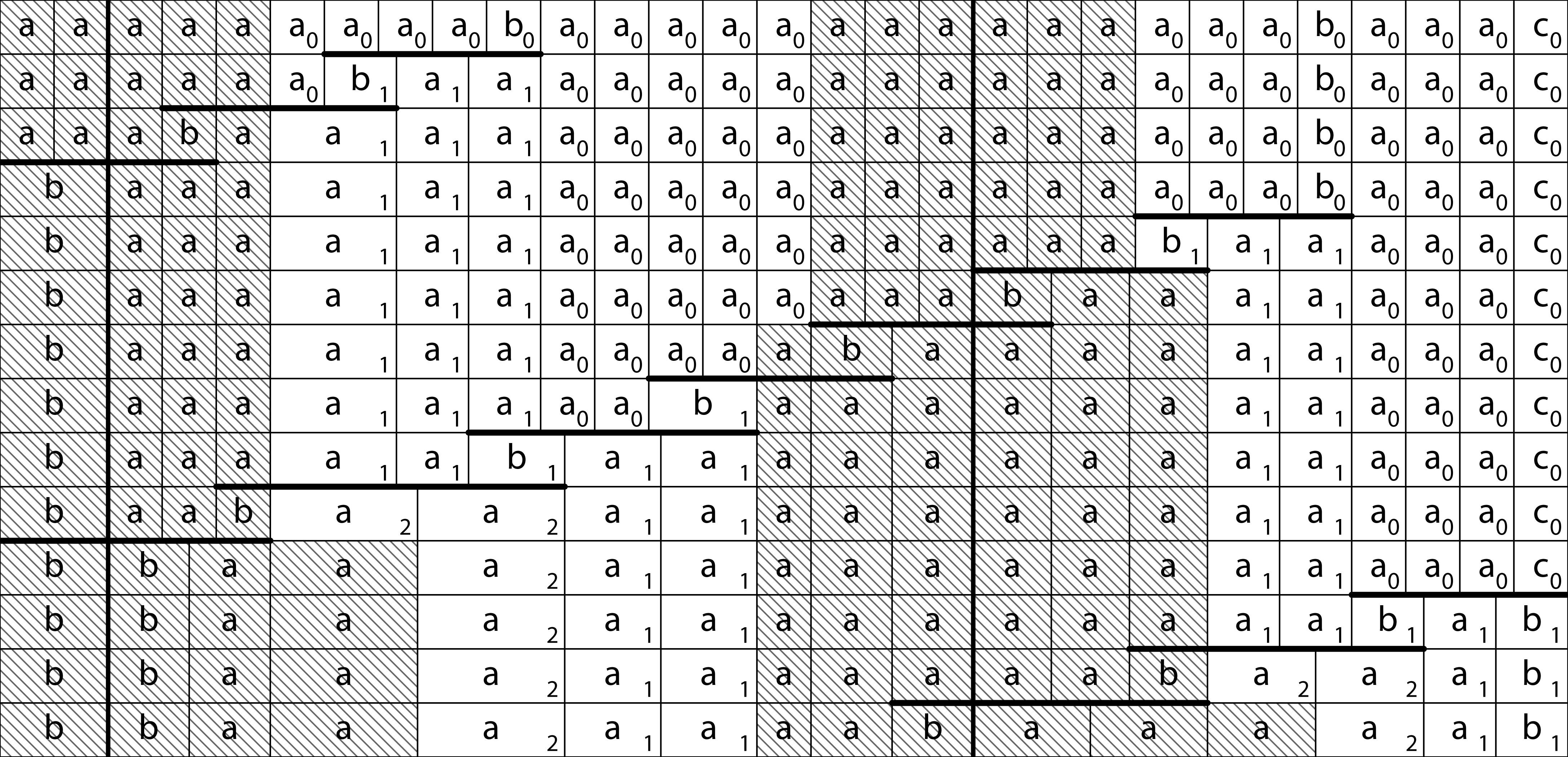}
\caption{A diagram for a history with subwords. Solid vertical lines
indicate subword boundaries. 
Cross-hatching indicates border letters. 
}
\label{subword}
\end{figure}

%Call the first and last $W-1$ letters of $w_i$, {\em border letters}.
%(If $w_0$ is at the start, or end, of $v_0$ then we need only designate
%the $W-1$ letters at the end, or start resp., of $w_i$ as border letters.)
%When a substitution takes place entirely inside $w_i$, and the \rhs\ 
%contains border letters, we make the widths of the \rhs\ border
%letters the same as those of the \lhs\ border letters, and expand any
%non-border letters to fill the space left. The statement and proof of
%Lemma~\ref{widerletter} go through just as before so long as we
%replace every occurrence of the word ``letter'' with the words
%``non-border letter''.

We restrict the definition of the generation of a letter to non-border
letters: the {\em generation} of a non-border letter in a \rhs\ is one
more than the least generation of the non-border letters of the
corresponding \lhs. With this adjustment, Lemma~\ref{width} goes
through. 

Splitting paths are defined as before except that we forbid any of the
words included in the detail to cross a subword boundary.
Lemma~\ref{pathexists} gives us such a path since we follow the edges
of non-border letters (letters for which the generation is
defined). Clearly a splitting path for $w_0\subseteq v_0$ is also one
for $v_0$.

\def\Wfrac{\left(\frac{W}{W-1}\right)}

We show that the number of splitting paths required to split all
subword histories, with a given starting length, is bounded by a
polynomial function of that starting length. This bound is independent
of the total number of substitutions in the history. It follows that
if we have enough histories, two of them will have equivalent
splitting paths.

\begin{lemma} \label{polysplits}
Let $w_0$ be a subword of $v_0$ of length $\ell(w_0) = N$, and  
let $w_0,\ldots,w_t$ be such that $\ell(w_t) \geq 2W-1$.
Then $w_t$ has a splitting path in one of at most $C_1N^{C_2}$
equivalence classes, where $C_1,C_2$ are positive constants depending
only on $|\A|$ and $W$. The splitting path can be chosen to
end next to any non-border letter of $w_t$. 
\end{lemma}

\begin{proof}
Since $\ell(w_t) \geq 2W-1$ it contains at least one non-border letter.
Choose one, and let $g$ be its generation. By Lemma~\ref{width}, its
width is at least $\Wfrac^g$. But since this cannot exceed $N=\ell(w_0)$, the
width of $w_0$, we have $g\leq \log_{\Wfrac}N$. 

By Lemma~\ref{pathexists}, we can find a splitting path, ending next
to our chosen letter, of length at most $2g+2$. By
Remark~\ref{numsplits}, the number of classes of splitting path, of
length $\leq 2g+2$, does not exceed $(2(W+1)^2 |\A|^{2W+1})^{2g+3}$.  Since
$g$ is bounded by a logarithm of $N$, the result follows.
\end{proof}

We consider now a word divided into two subwords $u_0,v_0$. We keep $v_0$
fixed, vary $u_0$, and run the algorithm for some amount of
time. Intuitively speaking, our algorithm carries information between
the two subwords, giving in principle a number of possible values for
$v_t$ which is exponential in $\ell(v_0)$. 
We show that the number of distinct $v_t$ that can actually arise is
only polynomial in $\ell(v_0)$.

\begin{lemma} \label{polynomial}
Let $v_0$ be a fixed word of length $N\geq 1$ in $\A^*$. For each word
$u_0$ we choose a time $t$ and let $u_t v_t$ be the
result of applying $t$ substitutions to $u_0 v_0$; $v_t$ is then a
function of $u_0$. There exist positive constants $C_0, C$, depending only on
$|\A|$ and $W$, such that $v_t(u_0)$ takes at most $C_0 N^C$ distinct
values as $u_0$ varies. The same bound applies if we instead define
$v_t u_t$ to be the result of applying $t$ substitutions to $v_0 u_0$. 
\end{lemma}

\def\rvtbar{[v_t^+]}

\begin{proof}
First consider all $v_t$ such that $\ell(v_t)\geq 2W-1$.  Then by
Lemma~\ref{polysplits}, each $v_0,\ldots,v_t$ has a splitting path
ending $W-1$ letters from the start of $v_t$, in one of at most
$C_1N^{C_2}$ classes.

Since $\rvtbar$ is determined by $v_0^+$ and the class of the splitting
path, and $v_0^+$ has at most $N$ possible values (each being a subword
at the end of $v_0$), $\rvtbar$ takes at most $C_1 N^{C_2+1}$ distinct
values as $u_0$ varies. Since $v_t^+$ is one of the, at most
$\ell(\rvtbar) \leq N$, words obtained by making substitutions in
$\rvtbar$, $v_t^+$ itself can take at most $C_1 N^{C_2+2}$ values.  On
the other hand, since $\ell(v_t^-) = W-1$, this can take at most
$|\A|^{W-1}$ values. Multiplying these two gives the required bound on
the number of values $v_t$ can take when $\ell(v_t)\geq 2W-1$. 

The number of possible words of length less than $2W-1$ is constant
and, since we are assuming $N\geq 1$, we can absorb this into $C_0$.

The proof for $v_t u_t$ is similar. 
\end{proof}

\begin{remark}
{\rm 
If the rewriting system were not required to delete a
letter with every substitution, the number of values $v_t(u_0)$ could
take might well be exponential in $\ell(v_0)$. 
}
\end{remark}

\section{Groups which have no \dehn Algorithm}\label{noDehn}

\def\half{\frac{1}{2}}

\noindent
In this section we use the results of the previous section to exhibit
groups which have no \dehn algorithm.

\begin{theorem} \label{expexp}

Let $G$ be a group with some fixed generating set.
Suppose that for each $n \ge 0$ there are sets $S_1(n) \subset
G$ and $S_2(n) \subset G$ satisfying the following:
\begin{enumerate}
\item For $i=1,2$ each element of $S_i(n)$ can be represented by a
word of exactly length $n$.
\item There are $\alpha_0>0$ and $\alpha_1>1$ so that for infinitely
many $n$
\[ |S_i(n)| \ge \alpha_0 \alpha_1^n.\]
\item Each element of $S_1$ commutes with each element of $S_2$.
\end{enumerate}
Then $G$ has no \dehn algorithm.

\end{theorem}

\begin{proof}

Suppose to the contrary that we have a deterministic \dehn algorithm
for $G$.  Let $\A$ be the working alphabet and let
$W$ be the length of the longest left hand side.
Choose $n>0$ such that $|S_i(n)| \ge \alpha_0
\alpha_1^n$, and 
\[\half\alpha_0\alpha_1^n > C_1n^{C_2 + 2} |\A|^{6W} C_0 n^C,\]
where $C,C_0,C_1$ and $C_2$ are as in Lemmas~\ref{polysplits}
and \ref{polynomial}. 
% The reason for this choice will become clear in
% the course of the proof. 
% Choose constants $\alpha_0,\alpha_{1}$ such that, for all $n\geq 0$ and $i
% = 1,2$, $|S_i(n)|\geq \alpha_0\alpha_1^n$.  
Let $T_i$ be a set of words of length $n$ representing
$S_i(n)$, for $i=1,2$.

% In particular $|T_i(n)|\geq \alpha_0\alpha_1^n$.

We shall consider the effect of our supposed \dehn algorithm on words
of the form $u_0 v_0 u_0^{-1} v_0^{-1}$, for $u_0 \in T_1$ and $v_0\in
T_2$.  All such words must reduce to the empty word since $S_1(n)$
and $S_2(n)$ commute.  Put $x_0 = u_0^{-1}$ and $y_0 = v_0^{-1}$, and
let $u_t v_t x_t y_t$ denote the result of applying $t$ substitutions
to $u_0 v_0 x_0 y_0$.

Define $t$, as a function of $u_0$ and $v_0$, to be the least integer
$t\geq 0$ such that $\max \{\ell(v_t), \ell(x_t)\} < 3W$. I.e., we run
the algorithm until the first time at which both $v_0$ and $x_0$ have
length less than $3W$.  Then $u_t, v_t, x_t$ and $y_t$ are all well
defined functions of $u_0$ and $v_0$.  Since at most $W$ letters are
deleted in each step it follows that $2W \leq \max \{\ell(v_t),
\ell(x_t)\} \leq 3W - 1$.

For each pair $(u_0,v_0) \in T_1\times T_2$, one or both of the
inequalities $\ell(v_t)\geq \ell(x_t)$, $\ell(x_t)\geq \ell(v_t)$
holds. Therefore one of these inequalities must hold for at least half
of $T_1\times T_2$. We shall suppose
it is the first, and argue to obtain a contradiction; were it the
second, a similar argument, interchanging the roles of $v_t$ and $x_t$,
would give a contradiction instead.

\vspace{0.7em}
{\em Step 1, fix a $v_0$:}
Since we are assuming that, for at least half the pairs $(u_0,v_0)$,
$\ell(v_t)\geq \ell(x_t)$, we can certainly find a $v_0\in T_2$
such that for at least half of $u_0\in T_1$,
$\ell(v_t(u_0,v_0))\geq \ell(x_t(u_0,v_0))$.  We fix this $v_0$ and
henceforth regard $u_t,v_t,x_t,y_t$ as functions of $u_0$ alone. Let
$U = \{u_0\in T_1\ |\ \ell(v_t)\geq \ell(x_t)\}$. 
By the choices we have made
have made, $|U|\geq \half |T_1| \geq \half\alpha_0\alpha_1^n$. 

\vspace{0.7em}
{\em Step 2, split the $v_t$ using boundedly many splitting classes:}
By the definitions of $t$ and $U$, $\ell(v_t) \geq 2W$ for each $u_0 \in
U$.  By Lemma~\ref{polysplits}, we can choose a splitting path for each
$v_0,\ldots,v_t$, using at most $C_1n^{C_2}$ classes.  If we take into
account also the position in $v_0$ at which the path begins, and the
position in $v_t$ at which the path ends, we get at most $C_1n^{C_2 +
2}$ classes.

\vspace{0.7em}
{\em Step 3, the map $u_0 \mapsto v_t x_t y_t$ is many-to-one:}
By the definition of $t$, $\ell(v_t), \ell(x_t)\leq 3W - 1$, so the
number of possible values $v_t x_t$ can take, as $u_0$ ranges over
$U$, is less than $|\A|^{6W}$. By Lemma~\ref{polynomial}, the
number of values $y_t$ can take is at most $C_0 n^C$, for positive constants
$C_0$ and $C$ depending only on $|\A|$ and $W$. 
On the other hand $|U| \geq \half\alpha_0\alpha_1^n$ so, by our choice 
of $n$, $|U| > C_1n^{C_2 + 2}
|\A|^{6W} C_0 n^C$. It follows that there exists a set of at least 
$C_1n^{C_2 + 2}$ $u_0$'s in $U$ which all give the same $v_t x_t y_t$. 

\vspace{0.7em}
{\em Step 4, construct a word that breaks the \dehn algorithm:}
From Step~3, we have more than $C_1n^{C_2 + 2}$ $u_0$'s giving the
same $v_t x_t y_t$. From Step~2, we have at most $C_1n^{C_2 + 2}$
positioned splitting path classes for $v_t$. Therefore we can find
$u_0, u_0'\in U$ such that (writing $v_t'$ for $v_t(u_0')$ etc.), $v_t x_t
y_t = v_t' x_t' y_t'$, and $v_0,\ldots,v_t$ and $v_0',\ldots,v_t'$
contain equivalent splitting paths, starting at the same position in
$v_0 = v_0'$ and ending at the same position in $v_t = v_t'$. 

By Lemma~\ref{splicing}, running the algorithm on $u_0 v_0^-
v_0^{\prime +} x_0' y_0'$ yields $u_t v_t^- v_t^{\prime +} x_t' y_t'$.
Now $u_0 v_0^- v_0^{\prime +} x_0' y_0' = u_0 v_0 x_0' y_0$, which
does not represent the identity in $G$ but rather $u_0 u_0^{\prime
-1}$. On the other hand, $u_t v_t^- v_t^{\prime +} x_t' y_t' = u_t
v_t x_t y_t$ which reduces to the empty word.  Therefore this
rewriting system does not implement a \dehn algorithm for $G$.
\end{proof}

In fact, it is not hard to strengthen this to the following.

\begin{theorem} \label{explin}

Let $G$ be a group with some fixed generating set.
Suppose that for each $n \ge 0$ there are sets $S_1(n) \subset
G$ and $S_2(n) \subset G$ satisfying the following:
\begin{enumerate}
\item For $i=1,2$ each element of $S_i(n)$ can be represented by a
word of exactly length $n$.
\item There are $\alpha_0>0$ and $\alpha_1>1$ and $\alpha_2>0$ so that
for all $n$
\[  |S_1(n)| \ge \alpha_0 \alpha_1^n\]
and
\[ |S_2(n)| \ge \alpha_2 n.\]
\item Each element of $S_1$ commutes with each element of $S_2$.
\end{enumerate}
Then $G$ has no \dehn algorithm.
\end{theorem}

\begin{proof}
The proof is very similar to that of Theorem~\ref{expexp} except that
we have to start with $v_0$ longer than $u_0$.  As before, suppose
that we have a deterministic \dehn algorithm for $G$, with working
alphabet $\A$, and longest \lhs\ of length $W$.
Choose $n_1$ and $n_2$ such that
\begin{equation} \label{enough.u}
\half \alpha_0\alpha_1^{n_1} > C_1n_2^{C_2 + 2} |\A|^{6W} C_0 n_2^C,
\end{equation}
and 
\begin{equation} \label{enough.v}
\half \alpha_2 n_2 > C_1 n_1^{C_2 + 2} |\A|^{6W} C_0 n_1^C.
\end{equation}
Let $T_i$ be a set of words of length $n_i$ bijecting to $S_i(n_i)$,
for $i=1,2$.

% Choose constants $\alpha_0,\alpha_1$ such that, for all $n_1\geq 0$,
% $|S_1(n_1)|\geq \alpha_0\alpha_1^{n_1}$. Choose $\alpha_2$ such that,
% for all $n_2\geq 0$, $|S_2(n_2)|\geq \alpha_2 n_2$. 

As before, let $u_t v_t x_t y_t$ denote the result of
applying $t$ substitutions to $u_0 v_0 u_0^{-1} v_0^{-1}$.
Define $t$, as a function of $u_0,v_0$,
such that $2W \leq \max \{\ell(v_t), \ell(x_t)\} \leq 3W - 1$.

If, for at least half of $(u_0,v_0) \in T_1\times T_2$,
$\ell(v_t) \geq \ell(x_t)$, we can argue as in Steps~1-4, using 
(\ref{enough.u}), to find $u_0,u_0'$ which break the algorithm. 
In the other case, arguing similarly, using (\ref{enough.v}), we can
find $v_0,v_0'$ which break the algorithm. 
\end{proof}

\begin{theorem}
Suppose $G$ has exponential growth and the center of $G$ contains an
infinite cyclic group.  Then $G$ has no \dehn algorithm.
\end{theorem}

\begin{proof}
We take among our generators for $G$ a letter $z$ denoting a central
element of infinite order and a letter $r$ denoting the identity.  We
can then Take $S_1(n) = B(n)$, for if $u$ is a geodesic denoting an
element of length $k \le n$, we can denote this element by
$ur^{n-k}$.  Thus each element of $B(n)$ is represented by a word of
length $n$.  Likewise, we can take $S_2 = \{z^k \mid |k| \le n \}$.
We then apply Theorem~\ref{explin}.
\end{proof}

\begin{corollary}
$F_2 \times \Z$ has no \dehn algorithm. \qed
\end{corollary}

\begin{corollary}
If $G$ is a braid group of 3 or more strands, $G$ has no \dehn
algorithm.  \qed
\end{corollary}

\begin{corollary}
If $M$ is a graph manifold one of whose pieces is a non-closed
Seifert fibered space, then $\pi_1(M)$ has no \dehn algorithm. \qed
\end{corollary}

\begin{corollary}
If $M$ is a closed 3-manifold modelled on either $\HTwo \times
\R$ or $PSL_2(\R)$, then $\pi_1(M)$ has no \dehn algorithm.  \qed
\end{corollary}

\begin{corollary}
Thompson's group $F$ has no \dehn algorithm.
\end{corollary}

\begin{proof}
Thompson's group $F$ has exponential growth \cite{ThGp} and contains a
subgroup isomorphic to the direct product of two copies of itself.
\end{proof}

We will say that a subgroup $A$ of $G$ has exponential growth in $G$
if $A \cap B(n)$ has exponential growth.

\begin{theorem}
Suppose $G$ has an abelian subgroup which has exponential growth in
$G$.  Then $G$ has no \dehn algorithm.
\end{theorem}

\begin{proof}
In this case we take $S_1(n)=S_2(n)=A \cap B(n)$ and apply
Theorem~\ref{expexp}. 
\end{proof}

\begin{corollary}
If $G$ is a Baumslag-Solitar group
\[\langle a,t \mid ta^pt^{-1} = a^q \rangle \]
with $p \ne \pm q$ then $G$ has no \dehn algorithm.
\end{corollary}

\begin{proof}
It is not hard to see that $\ell(a^n) = O(\ln n)$ and consequently,
$\langle a \rangle$ has exponential growth in $G$.
\end{proof}

\begin{corollary}
Suppose $M$ is a closed 3-manifold modelled on solvegeometry.  Then
$\pi_1(M)$ has no \dehn algorithm.
\end{corollary}

\begin{proof}
In this case $\pi_1(M)$ contains a finite index subgroup of the form
$A \rtimes \Z$ where $A$ is isomorphic to $\Z^2$ and the
action of the generator of $\Z$ has eigenvalues $\lambda$ and
$\lambda^{-1}$ with the modulus of $\lambda$ greater than 1.  It
follows that $A$ has exponential growth in $\pi_1(M)$.
\end{proof}

Combining these with our results on virtually nilpotent and
geometrically finite groups we have:

\begin{theorem}
Suppose $M$ is a graph manifold.  Then $\pi_1(M)$ has a \dehn
algorithm if and only if none of the following hold:
\begin{enumerate}
\item $M$ is closed $\HTwo \times \R$, $PSL_2(\R)$ or
solvegeometry manifold, or
\item $M$ has a non-closed Seifert fibered piece.
\end{enumerate}
\end{theorem}

\let\bibname\relax

%% \affiliationone{
%%   Oliver Goodman\\
%%   Department of Mathematics and Statistics\\
%%   University of Melbourne\\
%%   Parkville, Victoria 3052\\
%%   Australia}
%% \affiliationtwo{
%%   Michael Shapiro\\
%%   Department of Mathematics and Computer Science\\
%%   Rutgers University\\
%%   Newark, NJ\\
%%   USA}
\end{document}